\documentclass[10pt,reqno]{amsart}
\usepackage{amsmath}
\usepackage{amsfonts}
\usepackage{amssymb}
\usepackage{graphicx}
\usepackage{amsthm,graphicx,color,yfonts}
\usepackage{pdfsync}
\usepackage{epstopdf}%
\usepackage{pifont}
\usepackage{bbm}
\usepackage{a4wide}

\usepackage[colorlinks=true]{hyperref}
\hypersetup{linkcolor=red,citecolor=blue,filecolor=dullmagenta,urlcolor=blue}

\usepackage{wrapfig}
\usepackage{tikz}
\usetikzlibrary{arrows,calc,decorations.pathreplacing}
\definecolor{light-gray1}{gray}{0.90}
\definecolor{light-gray2}{gray}{0.80}
\definecolor{light-gray3}{gray}{0.60}

\def\norm#1{\|#1\|}
\def\bra#1{\langle#1\rangle}
\def\wt#1{\widetilde{#1}}
\def\wh#1{\widehat{#1}}

\def\set#1{\left\{#1\right\}}

\newcommand{\N}{{\mathcal{N}}}

\newcommand{\T}{{\mathbb T}}
\newcommand{\R}{{\mathbb R}}
\newcommand{\C}{{\mathbb C}}
\newcommand{\Z}{{\mathbb Z}}
\newcommand{\ft}{{\mathcal{F}}}

\newcommand{\M}{{\mathcal{M}}}

\newcommand{\Sch}{{\mathcal{S}}}
\newcommand{\supp}{{\mbox{supp}}}

\newcommand{\px}{\partial_x}
\newcommand{\pt}{\partial_t}

\newcommand{\la}{\lambda}

\numberwithin{equation}{section}

\newtheorem{theorem}{Theorem}[section]
\newtheorem*{theorem*}{Theorem}
\newtheorem{proposition}{Proposition}[section]
\newtheorem{corollary}{Corollary}[section]
\newtheorem{lemma}{Lemma}[section]
\newtheorem{definition}{Definition}[section]

\theoremstyle{remark}
\newtheorem{remark}{Remark}[section]

\title[IVP of the modified Kawahara]{Well-posedness issues on the periodic modified Kawahara equation}

\author{Chulkwang Kwak}
\address{Facultad de Matem\'aticas, Pontificia Universidad Cat\'olica de Chile and Institute of Pure and Applied Mathematics, Chonbuk National University}
\email{chkwak@mat.uc.cl}
\thanks{C. K. is supported by FONDECYT Postdoctorado 2017 Proyecto No. 3170067 and project France-Chile ECOS-Sud C18E06.}


\begin{document}
\begin{abstract}
This paper is concerned with the Cauchy problem of the modified Kawahara equation (posed on $\T$), which is well-known as a model of capillary-gravity waves in an infinitely long canal over a flat bottom in a long wave regime \cite{Hasimoto1970}. We show in this paper some well-posedness results, mainly the \emph{global well-posedness} in $L^2(\T)$. The proof basically relies on the idea introduced in Takaoka-Tsutsumi's works \cite{TT2004, NTT2010}, which weakens the non-trivial resonance in the cubic interactions (a kind of smoothing effect) for the local result, and the global well-posedness result immediately follows from $L^2$ conservation law. An immediate application of Takaoka-Tsutsumi's idea is available only in $H^s(\T)$, $s > 0$, due to the lack of $L^4$-Strichartz estimate for arbitrary $L^2$ data, a slight modification, thus, is needed to attain the local well-posedness in $L^2(\T)$. This is the first low regularity (global) well-posedness result for the periodic modified Kwahara equation, as far as we know. A direct interpolation argument ensures the \emph{unconditional uniqueness} in $H^s(\T)$, $s > \frac12$, and as a byproduct, we show the weak ill-posedness below $H^{\frac12}(\T)$, in the sense that the flow map fails to be uniformly continuous. 
\end{abstract}

\maketitle

\tableofcontents

\section{Introduction}\label{sec:intro}
\subsection{Setting}

A study on waves starts from an examination of a two-dimensional, irrotational flow of an incompressible ideal fluid with a free surface under the gravitational field. The fluid is bounded below by a solid bottom and above by an atmosphere of constant pressure. The upper surface is a free boundary, and the influence of the surface tension is naturally taken into account on the free surface. The motion of the free surface is called a capillary-gravity wave, and it is called a gravity wave or a water wave in the case without the surface tension.

\medskip

In the mathematical view, the waves are formulated as a free boundary problem for the incompressible, irrotational Euler equation. Rewriting the equations in an appropriate non-dimensional form, one gets two non-dimensional parameters $\delta := \frac{h}{\lambda}$ and $\varepsilon := \frac{a}{h}$, where $h, \lambda$ and $a$ denote the water depth, the wave length and the amplitude of the free surface, respectively, and another non-dimensional parameter $\mu$ called the Bond number, which comes from the surface tension on the free surface. The physical condition $\delta \ll 1$ characterizes the waves, which are called long waves or shallow water waves, but there are several long wave approximations according to relations between $\varepsilon$ and $\delta$. We introduce three typical long wave regimes.

\medskip

\begin{enumerate}
\item Shallow water wave: $\varepsilon = 1$ and $\delta \ll 1$.
\item Korteweg-de Vries (KdV): $\varepsilon = \delta^2 \ll 1$ and $\mu \neq \frac13$.
\item Kawahara: $\varepsilon = \delta^4 \ll 1$ and $\mu = \frac13 + \nu\varepsilon^{\frac12}$.
\end{enumerate}

\medskip

In Item (1) regime, we obtain the (so-called) shallow water equations as the limit $\delta \to 0$. It is known that the shallow water equations are analogous to one-dimensional compressible Euler equations for an isentropic flow of a gas of the adiabatic index $2$, and thus its solutions generally have a singularity in finite time, even if the initial data are sufficiently smooth. Therefore, this long wave regime is used to explain breaking waves of water waves. In Item (2) regime, the following well-known, notable equation called the KdV equation has beed derived from the equations for capillary-gravity waves by Korteweg and de Vries \cite{KdV1895}: 
\[\pm2 u_t + 3uu_x +\left( \frac13 - \mu\right)u_{xxx} = 0.\]
Remark that when the Bond number $\mu = \frac13$, this equation degenerates to the inviscid Burgers equation. In connection with this critical Bond number, Hasimoto \cite{Hasimoto1970} derived a higher-order KdV equation of the form
\[\pm2 u_t + 3uu_x - \nu u_{xxx} +\frac{1}{45}u_{xxxxx} = 0.\]
in Item (3) regime, which is nowadays called the Kawahara equation.

\bigskip

This paper concerns with the Cauchy problem of the modified Kawahara equation given by\footnote{The equation \eqref{eq:kawahara} is reduced from the equation of the form 
\[\pt v + \alpha \px^5 v + \beta \px^3 v + \gamma \px v + \mu\px (v^3) = 0\]
 by  the renormalization of $v$.}
\begin{equation}\label{eq:kawahara}
\begin{cases}
\pt v - \px^5 v + \beta \px^3 v + \gamma \px v - \frac{\mu}{3}\px (v^3) = 0 ,\\
v(t,x) = v_0(x) \in H^s(\T),
\end{cases}
(t,x) \in [0,T] \times \T,
\end{equation}
where $\T = \R /2\pi \Z$, $\beta \ge 0$\footnote{One may extend the range of $\beta$ to negative values, but this change does not cause any difficulty in our analysis by regarding the case $|\xi| \le \frac{6|\beta|}{5}$ as low frequency part.}, $\gamma \in \R$, $\mu = \pm 1$ and $u$ is a real-valued unknown.

\medskip

The equation \eqref{eq:kawahara} can be generalized as follows:
\begin{equation}\label{eq:gkawahara}
\pt v - \px^5 v + \beta \px^3 v + \gamma \px v - \frac{\mu}{3}\px (v^p) = 0, \quad p = 2,3, \cdots.
\end{equation}

As seen before, when $p=2$, the equation \eqref{eq:gkawahara} is the Kawahara equation, which is a higher-order Korteweg-de Vries (KdV) equation with an additional fifth-order derivative term. This type of the equation \eqref{eq:gkawahara} was first found by Kakutani and Ono \cite{KO1969} in an analysis of magnet-acoustic waves in a cold collision free plasma. The equation \eqref{eq:gkawahara} was also derived, as already mentioned above, by Hasimoto \cite{Hasimoto1970} as a model of capillary-gravity waves in an infinitely long canal over a flat bottom in a long wave regime when the Bond number is nearly $\frac13$. Kawahara \cite{Kawahara1972} studied this equation \eqref{eq:gkawahara} numerically and observed that the equation has both oscillatory and monotone solitary wave solutions. This equation is also regarded as a singular perturbation of KdV equation. We further refer to, for instance, \cite{Benney1977, AS1985, Yamamoto1986, HS1988, PRG1988, Boyd1991, GJ1995, Iguchi2007, SW2012} and references therein for more background informations.

\medskip

The equation \eqref{eq:kawahara} admits at least three conservation laws:
\[E[v](t) = \int v \; dx = M_0[v_0],\]
\begin{equation}\label{eq:L2}
M[v](t) = \frac12 \int v^2 \; dx = M[v_0]
\end{equation}
and
\begin{equation}\label{Hamiltonian}
H[v](t) =  \frac{1}{2} \int (\px^2 v)^2 \;dx + \frac{\beta}{2} \int (\px v)^2 \; dx -\frac{\gamma}{2} \int  v^2 \; dx +\frac{\mu}{12} \int v^4 \; dx = H[v_0].
\end{equation}
The $L^2$ conserved quantity \eqref{eq:L2} will help us to extend the local solution to global one, so to attain the global well-posedness in $L^2(\T)$. Moreover, the equation \eqref{eq:kawahara} can be written as the Hamiltonian equation with respect to \eqref{Hamiltonian} as follows:
\begin{equation}\label{eq:Hamiltonian form}
v_t = \partial_x \nabla_v H\left(v\left(t\right)\right)  = \nabla_{\omega_{-\frac{1}{2}}} H\left(v\left(t\right)\right),
\end{equation}
where $\nabla_v$ is the $L^2$ gradient and $\omega_{-\frac{1}{2}}$ is the symplectic form in $H^{-\frac{1}{2}}$ defined as 
\[\omega_{-\frac{1}{2}}\left(v,w\right) := \int_{\T} v \partial_x^{-1} w dx,\]
for all $u,v \in H_0^{-\frac{1}{2}}$. Indeed, a direct computation yields
\[\begin{aligned}
\omega_{-\frac{1}{2}}\left(w, \nabla_{\omega_{-\frac{1}{2}}}H\left(v\left(t\right)\right)\right):=&~{} \left.\frac{d}{d\varepsilon}\right|_{\varepsilon=0} H \left(v+\varepsilon w\right) \\
=&~{}\int \left(\px^4v - \beta \px^2 v - \gamma v + \frac{\mu}{3} v^3 \right) w \\
=&~{}\int \left(-\px^5v + \beta \px^3 v + \gamma \px v - \frac{\mu}{3} \px \left(v^3\right) \right) \px^{-1} w \\
=&~{}\omega_{-\frac{1}{2}} \left(w,\px^5v - \beta \px^3 v - \gamma \px v + \frac{\mu}{3} \px \left(v^3\right)\right).
\end{aligned}\]
Such an expression by the Hamiltonian form (or regarding the flow map as a symplectomorphism on $H^{-\frac12}$) enables to study some symplectic property, in particular, non-squeezing property, which is initiated (for the dispersive PDE or a non-compact operator) by Bourgain \cite{Bourgain1994}. However, our well-posedness result presented below is available only up to in $L^2$ regularity level, thus we cannot explore it as of now. More delicate analysis (or new clever idea) will facilitate the $H^{-\frac12}$ global well-posedness, and so the non-squeezing analysis. We refer to \cite{Kuksin1995, CKSTT2005, Roumegoux2010, Mendelson2014, HK2016, KVZ2016-1, KVZ2016-2, Kwak2018-1} for more detailed expositions of the non-squeezing property.

\medskip

On the other hand, the expression \eqref{eq:Hamiltonian form} provides a convenient setting to use the spectral stability theory of \cite{DT2017}. We also refer to \cite{TDK2018} for another application of the Hamiltonian form \eqref{eq:Hamiltonian form} to derive criteria for instability of small-amplitude periodic solutions of \eqref{eq:Hamiltonian form}.

\medskip

These conserved quantities play important roles in the study of the partial differential equations. In particular, such conserved quantities enable to treat the (nontrivial) resonant interaction in the study of the initial value problem under the periodic boundary condition. In this work, the second conserved quantity \eqref{eq:L2} is enough to deal with the cubic resonance, since the nonlinearity in \eqref{eq:kawahara} has only one derivative and is of the cubic form. On the other hand, an appropriate nonlinear transformation, which has a bi-continuity property, helps to remove the cubic nontrivial resonance without using the conservation law \eqref{eq:L2} (see Section \ref{sec:pre} for more details, and refer to \cite{Staffilani1997, Kwak2016} for similar or more complicate cases).

\subsection{Different phenomena: periodic vs. non-periodic}\label{sec:Diff Phen}

The Cauchy problems for some dispersive equations have plenty of interesting issues under the periodic setting compared with the non-periodic problems. The first interesting (and also different from the non-periodic problem) issue is the presence of non-trivial resonances. In particular, the modified Kawahara equation \eqref{eq:kawahara} contains two non-trivial resonant terms (of the Fourier coefficient forms) in the nonlinearity such as
\[n|\wh{v}(n)|^2 \wh{v}(n) \quad \mbox{and} \quad n \left(\sum_{n' \in \Z} |\wh{v}(n)|^2 \right)\wh{v}(n),\]
whenever the Fourier variables have the following frequency relations:
\[n_1 +n _2 + n_3 = n \quad \mbox{and} \quad (n_1+n_2)(n_2+n_3)(n_3+n_1) =0.\]

\medskip

The latter resonance causes an uncontrollable perturbation phenomenon near the linear solution in the Sobolev space (of any regularity), while the former one is controllable perturbation (at least up to $H^{\frac12}$). Such phenomena never happen under the non-periodic condition, since this happens on the set of frequencies ($n_1, n_2$ and $n_3$), for which elements satisfy $(n_1+n_2)(n_2+n_3)(n_3+n_1) =0$. To deal with the second resonance, one can use $L^2$ conservation law \eqref{eq:L2} to make $\left(\sum_{n' \in \Z} |\wh{v}(n)|^2 \right)$ as a constant coefficient of the first order linear term, and thus remove this resonance in the nonlinearity. 

\medskip

On the other hand, the first resonance is more difficult to be dealt with, precisely, one cannot make it a constant coefficient linear part of the equation unlike the second one. However, as mentioned above, this resonance does not make any trouble in the study on the well-posedness problem up to $H^{\frac12}$ regularity. To lower the regularity (in other words, to study the IVP with rougher data), it is necessary to take a more delicate analysis on this term. When studying on this term, we face on another interesting issue under the periodic setting: the lack of smoothing effect. Possible remedies to this problem are for instance, the normal form reduction method and the short time Fourier restriction norm method. In the present paper, we take the normal form mechanism to gain a smoothing effect under non-resonant interactions. A better example to capture this difference is the fifth-order modified KdV equation \cite{Kwak2018-2}.

\smallskip

For more detailed expositions, see Sections \ref{sec:pre} and \ref{sec:energy}.

\subsection{Main results}
Before stating our main result, we introduce a well-known notion of well-posedness. The Duhamel principle ensures that the equation \eqref{eq:kawahara} is equivalent to the following integral equation
\begin{equation}\label{Duhamel}
v(t) = S(t)v_0 + \frac{\mu}{3}\int_0^t S(t-s) \left(v(s)^3\right)_x \; ds,
\end{equation}
where $S(t)$ is the linear propagator associated to the linear equation $\pt v - \px^5 v + \beta \px^3 v + \gamma \px v = 0$, precisely defined by
\[S(t)f = \frac{1}{2\pi} \sum_{n \in \Z} e^{inx}e^{it(n^5 + \beta n^3 - \gamma n)} \wh{f}(n).\]

\medskip

The equation 
\begin{equation}\label{eq:kawahara_simple}
\pt v - \px^5 v + \px(v^3)= 0
\end{equation} 
allows the scaling invariance, that is, if $v$ is a solution to \eqref{eq:kawahara_simple}, then $v_{\lambda} : = \lambda^{-2} v (\lambda^{-5} t, \lambda^{-1} x)$, $\lambda \ge 1$, is also a solution to \eqref{eq:kawahara_simple}. A straightforward calculation gives
\begin{equation}\label{eq:scaling argument}
\norm{f_\lambda}_{\dot{H}^s(\T_{\lambda})} = \lambda^{-\frac32-s}\norm{f}_{\dot{H}^s(\T)},
\end{equation}
which says $s_c = -\frac32$ is (scaling) critical Sobolev index, where $\T_{\lambda} = \R / (2\pi\lambda \Z)$.
\begin{remark}
One can not get the scaling invariance for the equation \eqref{eq:kawahara} due to $\px^3 v$ and $\px v$ terms. Instead, one sees that the equation \eqref{eq:kawahara} allows a scaling equivalence, that is to say, if $v$ is solution to \eqref{eq:kawahara}, then $v_{\lambda} : = \lambda^{-2} v (\lambda^{-5} t, \lambda^{-1} x)$ is a solution to
\begin{equation}\label{eq:scaled kawahara}
\pt v_{\lambda} - \px^5 v_{\lambda} + \beta \lambda^{-2} \px^3 v_{\lambda} + \gamma \lambda^{-4}\px v_{\lambda} - \frac{\mu}{3}\px (v_{\lambda}^3) = 0.
\end{equation}
However, even the smaller dispersive effect from $\px^3 v$ and $\px v$ itself is negligible compared with one from $\px^5 v$ (and thus no influence on our analysis), hence the equation \eqref{eq:kawahara} follows the scaling rule observed above.
\end{remark}

\begin{remark}
The scaling symmetry is necessary to prove the local-well-posedness in $L^2$, in particular, for the validity of $L^4$ estimate.
\end{remark}

\medskip

We first state well-known definition of the local well-posedness for (scaling sub-critical) IVPs (see, for instance \cite{Cazenave2003, Tao2006}).
\begin{definition}[Local well-posedness]\label{def:WP}
Let $v_0 \in H^s(\T)$ be given. We say that the IVP of \eqref{eq:kawahara} is locally well-posed in $H^s(\T)$ if the following properties hold:
\begin{enumerate}
\item \emph{(Existence)} There exist a time $T = T(\norm{v_0}_{H^s}) > 0$ and a solution $v$ to \eqref{eq:kawahara} such that $v$ satisfies \eqref{Duhamel} and belongs to a subset $X_T^s$ of $C([0,T];H^s)$.

\medskip
 
\item \emph{(Uniqueness)} The solution is unique in $X_T^s$.

\medskip 

\item \emph{(Continuous dependence on the data)} The map $v_0 \mapsto v$ is continuous from a ball $B \subset H^s$ to $X_T^s$ (with the $H^s$ topology).
\end{enumerate}
\end{definition}

\begin{remark}\label{rem:WP}
We extend the notion of well-posedness presented in Definition \ref{def:WP} in (at least) three directions.
\begin{enumerate}
\item (Global well-posedness) We say the IVP is \emph{globally} well-posed if we can take $T$ arbitrary large.

\medskip

\item (Unconditional uniqueness) We say that the IVP is \emph{unconditional} well-posed if we can take $X_T^s = C([0,T];H^s)$.

\medskip

\item (Uniform well-posedness) We say that the IVP is \emph{uniform} well-posed if the solution map $v_0 \mapsto v$ is uniformly continuous from a ball $B \subset H^s$ to $X_T^s$. Similarly, one can define the notion of Lipschitz well-posedness, $C^k$ well-posedness, $k=1,2,\cdots$, and analytic well-posedness.
\end{enumerate}
\end{remark}

\begin{remark}\label{rem:ill}
Once the Picard iteration method works well on a IVP, one immediately obtains that the map is not only uniformly continuous but also real analytic (in this case, we say the problem a \emph{semilinear} problem). On the other hand, if one cannot apply the iteration method to a IVP (due to, for example, a strong nonlinearity compared to a dispersion or the presence of non-trivial resonances), one cannot reach the uniform well-posedness. This case is referred to as \emph{weakly} or \emph{mild} ill-posedness (in this case, we say the problem a \emph{quasilinear} problem).
\end{remark}

We are now in a position to state results established in this paper. The first theorem is to show the uniform well-posedness of \eqref{eq:kawahara}.
\begin{theorem}\label{thm:local0}
Let $s \ge \frac12$. Then, the Cauchy problem of \eqref{eq:kawahara} is locally (in time) well-posed in $H^s(\T)$. Moreover, the uniform continuity (indeed analytic) of the flow map holds in the class 
\begin{equation}\label{eq:class}
\set{v_0 \in H^s : \norm{v_0}_{L^2} = c}, \quad c \ge 0 \quad \mbox{is fixed}.
\end{equation}
\end{theorem}
The proof is based on the standard Fourier restriction norm method (with trilinear estimates), initially introduced by Bourgain \cite{Bourgain1993}. The regularity threshold $s = \frac12$, for which the local well-posedness of \eqref{eq:kawahara} holds, occurs due to the nontrivial resonant term $i\mu n|\wh{v}(n)|^2\wh{v}(n)$ as explained in Section \ref{sec:Diff Phen}. To improve Theorem \ref{thm:local0} below $H^{\frac12}(\T)$, it is necessary to reduce the strength of the resonance. 

\medskip

In \cite{TT2004, NTT2010}, the authors introduced a way to weaken the resonance by establishing a kind of smoothing effects in the context of modified KdV equation. Let us be more precise (in the context of modified Kawahara equation \eqref{eq:kawahara}). The evolution operator given by
\[\mathcal{W}(t)v := \frac{1}{2\pi} \sum_{n \in \Z} e^{i(nx + tp(n) +  t n|\wh{v}_0(n)|^2)}\wh{v}(n)\]
enables to reduce $in|\wh{v}(n)|^2\wh{v}(n)$ by $in\left(|\wh{v}(n)|^2 - \wh{v}_0(n)|^2\right)\wh{v}(n)$. An appropriate estimate for 
\[\sup_{n \in \Z} \left|\mbox{Im}\left[ \int_0^t  n^2 \sum_{\N_n}\wh{v}(s,n_1)\wh{v}(s,n_2)\wh{v}(s,n_3)\wh{v}(s,-n) \; ds \right]\right|\]
succeeds in getting a kind of smoothing effects, and so the local well-posedness below $H^{\frac12}(\T)$. However, an immediate application of the argument in \cite{TT2004, NTT2010} does not ensure the smoothing effect in $L^2$ level, due to, as a technical reason, the lack of $L^4$ estimate for $L^2$ data. To resolve this problem we employ the scaling argument to make $L^2$ data sufficiently small (thanks to \eqref{eq:scaling argument}). See Section \ref{sec:pre} for more details. We state the main result in this paper.

\begin{theorem}\label{thm:local}
Let $0 \le s < \frac12$. Then, the Cauchy problem of \eqref{eq:kawahara} is locally (in time) well-posed in $H^s(\T)$.
\end{theorem}
Thanks to $L^2$ conservation law \eqref{eq:L2}, we extend Theorem \ref{thm:local} to the global one.
\begin{theorem}\label{thm:global}
The Cauchy problem of \eqref{eq:kawahara} is globally (in time) well-posed in $L^2(\T)$.
\end{theorem}

In the proof of Theorem \ref{thm:local}, we are not able to attain the uniform well-posedness even in the class \eqref{eq:class} due to the resonance $in|\wh{v}(n)|^2 \wh{v}(n)$, that is, the flow map defined in the proof of Theorem \ref{thm:local} for $0 \le s < \frac12$ does not satisfy the uniform continuity in the following sense:
\begin{definition}[Uniform continuity of flow maps]\label{def:UCD}
We say that the flow map is uniformly continuous if for all $R>0$, there exist $T > 0$ and a continuous function $\zeta$ on $[0, \infty)$ satisfying $\zeta(r) \to 0$ as $r \to 0$ such that solutions $v_1, v_2$ to \eqref{eq:kawahara} with $\norm{v_1(0)}_{H^s}, \norm{v_2(0)}_{H^s} \le R$ satisfy
\[\norm{v_1 - v_2}_{C_TH^s} \le \zeta(\norm{v_1(0) - v_2(0)}_{H^s}).\]
\end{definition}
As a byproduct of Theorem \ref{thm:local}, we have
\begin{theorem}\label{thm:weak ill-posedness}
Let $0 \le s < \frac12$. Then, the Cauchy problem of \eqref{eq:kawahara} is weakly ill-posed in $H^s(\T)$ in a sense of Remark \ref{rem:ill}. In other words, the flow map does not hold the property presented in Definition \ref{def:UCD}.
\end{theorem}

An interesting issue in the well-posedness theory is the \emph{unconditional uniqueness} as mentioned in Remark \ref{rem:WP}, that is to say, the uniqueness holds in some larger spaces that contain weak solutions even in higher regularity. Such an issue was first proposed by Kato \cite{Kato1995} in the context of Schr\"odinger equation. The unconditional uniqueness is referred to as the uniqueness in $L_T^{\infty}H^s$ without the restriction of any auxiliary function space (for instance $X_T^{s,\frac12}$\footnote{The function space $X_T^{s,b}$ is the time localization of the standard $X^{s,b}$ introduced by Bourgain \cite{Bourgain1993}} used in Theorem \ref{thm:local}).

\medskip

The unconditional well-posedness of some dispersive equations have been studied (for instance \cite{Kato1995, Zhou1997, FPT2003, Win2008, Tao2009, BIT2011, KO2012, GKO2013, MPV2018, MPV1, KOY2019, MY2018} and references therein). Some of these uniqueness results employed some auxiliary function spaces (for example Strichartz spaces \cite{Tao2009}, $X^{s,b}$-type \cite{Zhou1997, Win2008, MPV2018, MPV1}), which are designed to be large enough to contain $C_TH^s$ such that the uniqueness of the solution holds. On the other hand, a straightforward energy-type estimate via finite or infinite iteration scheme of the normal form reduction method is also available to prove the unconditional well-posedness in a certain class of $C_TH^s$ \cite{BIT2011, KO2012, GKO2013, KOY2019, MY2018}. Such an argument seems more natural and elementary since any other auxiliary function spaces does not be taken.

\medskip

We finally state the last result established in this work. The uniqueness in $X_T^{0,\frac12}$ established in Theorems \ref{thm:local} and \ref{thm:global} ensures 
\begin{theorem}\label{thm:unconditional}
Let $s > \frac12$. Then, the Cauchy problem of \eqref{eq:kawahara} is unconditionally globally well-posed in $H^s(\T)$.
\end{theorem}

For the proof of Theorem \ref{thm:unconditional}, we show the embedding property $C_TH^s \subset X_T^{0,\frac12}$, $s > \frac12$, where the space $X_T^{0,\frac12}$ was designed for the proof of Theorem \ref{thm:local}. Thus, the uniqueness in $X_T^{0,\frac12}$ in addition to Theorem \ref{thm:global} implies the unconditional well-posedness.

\begin{remark}
An alternative way used in \cite{KO2012} seems available to prove the unconditional well-posedness in $H^s(\T)$, $s > \frac12$ (also possible in $H^{\frac12}(\T)$). However, we do not take their argument for the proof of Theorem \ref{thm:unconditional} in order to avoid abusing the normal form mechanism. 
\end{remark}

\begin{remark}
Theorems \ref{thm:weak ill-posedness} and \ref{thm:unconditional} will be improved in the forthcoming work by developing the argument inspired by, for instance, Okamoto \cite{Okamoto2017} and Molinet, Pilod and Vento \cite{MPV2018}, respectively.
\end{remark}

\subsection{About the proof of Theorem \ref{thm:local}}

The proof is the standard compactness argument (or referred to as energy method). After changing of suitable variables ($v \mapsto v_{\la} =: u$) and collecting the non-resonance estimate (Lemma \ref{lem:nonres-trilinear}) and a smoothing effect (Corollary \ref{cor:main}), one establishes
\[\begin{aligned}
\norm{u}_{X_{\la}^{0,\frac12}} \le&~{} C_0\norm{u_0}_{L_{\la}^2} + C_1\la^{\frac12-}\norm{u}_{X_{\la}^{s,\frac12}}^3 \\
&+C_1\la^{0+} \Bigg(\norm{u_0}_{L_{\la}^2}^4 + \big(\norm{u_0}_{L_{\la}^2}^2 + \norm{u}_{X_{\la}^{0,\frac12}}^4\big)^2 +  (1+ \norm{u}_{X_{\la}^{0,\frac12}}^2)\norm{u}_{X_{\la}^{0,\frac12}}^3\Bigg) \norm{u}_{X_{\la}^{0,\frac12}}
\end{aligned}\]
This guarantees the uniform boundedness of $u$ in $X_{\la}^{s,\frac12}$ dependent only on the initial data for sufficiently large but fixed $\la \gg 1$, and thus the weak and strong convergence of $u_j$ to $u$ in $H_{\la}^s$ can be attained via Arzel\`a-Ascoli compactness theorem. $L^2$ conservation law \eqref{eq:L2} extends the local solution $u$ to the global one.  

\subsection{About the literature}

Not only the modified Kawahara equation, but also the generalized Kawahara equation (including the Kawahara, $p=2$ in \eqref{eq:gkawahara}) have been extensively studied in several directions. 

\medskip

There has been a great deal of work on solitary wave solutions of the Kawahara equation in the last fifty years. Compared to the KdV solitary waves, the Kawahara solitary wave solutions exponentially decay to zero as $x \to \infty$ analogously to KdV, while, the Kawahara solitary waves have oscillatory trails, unlike the KdV equation whose solitary waves are non-oscillating. The strong physical background of the Kawahara equation and such similarities and differences between Kawahara and the KdV equations in both the formulations, and the behavior of the solutions propound the mathematical interesting questions of this equation. We refer to, for instance, \cite{Kawahara1972, GP1977, AS1985, HS1988, PRG1988, Boyd1991, KB1991, IS1992, KO1992, Biswas2009} for more informations associated to solitary waves of \eqref{eq:gkawahara} and \cite{Levandosky1999, Levandosky2007, Natali2010, TDK2018, KM2018} for their stability results.

\medskip

The Cauchy problems for the Kawahara and modified Kawahara equations posed on $\R$ have been extensively studied. For the Kawahara equation (\eqref{eq:gkawahara} with $p=2$), we refer to \cite{CT2005, CDT2006, WCD2007, CLMW2009, JH2009, CG2011, Kato2011, Kato2013, Okamoto2017} for the well- and ill-posedness results. As the best result in the sense of the low regularity Cauchy problem, Kato \cite{Kato2011, Kato2013} proved the local well-posedness for $s \ge -2$ by modifying $X^{s,b}$ space, the global well-posedness for $s > -\frac{38}{21}$ and the ill-posedness for $s<-2$ in the sense that the flow map is discontinuous at zero. Recently, Okamoto \cite{Okamoto2017} observed the norm inflation with general initial data, which implies that the flow map of the Kawahara equation is discontinuous everywhere in $H^s(\R)$ with $s < -2$.

\medskip

For the modified Kawahara equation (\eqref{eq:gkawahara} with $p=3$), we refer to \cite{JH2009, CLMW2009, YLY2011, YL2012} for the  the local well-posedness in $H^s(\R)$, $s \ge -1/4$, the global well-posedness in $H^s(\R)$, $s>-3/22$, and the weak ill-posedness in $H^s(\R)$, $s < -\frac14$.

\medskip

Compared with above well-posedness results for the non-periodic problems, there is a few work on the Cauchy problems under the periodic boundary condition. Gorsky and Himonas \cite{GH2009} have first studied the higher-order KdV-type equation of the form 
\[u_t + u_{mx} + uu_x = 0, \quad m=3,5,7, \cdots\]
under the periodic boundary condition. They established the bilinear estimate in $X^{s,\frac12}$, $s \ge -\frac12$ to prove the local well-posedness in $H^{-\frac12}(\T)$. This result was improved by Hirayama \cite{Hirayama2012}. He improved the bilinear estimate established in \cite{GH2009} in $H^s(\T)$ level, $s \ge -\frac{m-1}{4}$ to show the local well-posedness $H^{-\frac{m-1}{4}}$, and this estimate was shown to be sharp in the standard $X^{s,b}$. The global extension of this result was done by Hong and the author \cite{HK2016} via \emph{I-method}. The optimal local well-posedness result in $H^s(\T)$, $s \ge -\frac32$, for the Kawahara equation has been established by Kato \cite{Kato2012} by constructing a modified $X^{s,b}$ space (motivated by Bejenaru and Tao \cite{BT2006}) in order to handle the strong nonlinear interactions appeared when $s < -1$. He also proved the $C^3$-ill-posedness when $s < -\frac32$. 

\subsection*{Organization of this paper} The rest of the paper is organized as follows: In Section \ref{sec:pre}, we give a fundamental observation to study \eqref{eq:kawahara}, and introduce (modified) Takaoka-Tsutsumi's idea adapted to this problem. We also introduce $X^{s,b}$ space and its properties, and provide essential lemmas for the rest of sections. In Section \ref{sec:tri}, we prove the standard trilinear estimates in $X^{s,b}$. In section \ref{sec:energy}, we prove a smoothing property to control the reduced resonance below $H^{\frac12}$, and thus we prove local and global well-posedness results in Section \ref{sec:global}. In Section \ref{sec:unconditional}, as an application of local well-posedness in $L^2$, we show the \emph{unconditional uniqueness} of weak solutions to \eqref{eq:kawahara} in $H^s(\T)$, $s > \frac12$. In Appendices, we provide the proof of $L^4$ Strichartz estimate and a short proof of Theorem \ref{thm:weak ill-posedness} for the sake of the reader's convenience.

\subsection*{Acknowledgments} 
The author would like to thank Soonsik Kwon for a helpful discussion and comments on the notion and well-known argument of \emph{unconditional uniqueness}. Part of this work was complete while the author was visiting KAIST (Daejeon, Republic of Korea) and Chung-Ang University (Seoul, Republic of Korea). The author acknowledges the warm hospitality of both institutions. 

\medskip

The author would like to express his gratitude to anonymous referees for careful reading and the valuable comments, in particular, pointing out an error in the use of a nonlinear transform in the earlier version of manuscript.

\section{Preliminaries}\label{sec:pre}

\subsection{Notations}
Let $x, y \in \R_+$. We use $\lesssim$ when $x \le Cy$ for some $C >0$. Conventionally, $x \sim y$ means $x \lesssim y$ and $y \lesssim x$. $x \ll y$, also, denotes $x \le cy$ for a very small positive constant $c > 0$.

\medskip

Let $f \in \Sch '(\R \times \T) $ be given. $\wt{f}$ or $\ft (f)$ denotes the space-time Fourier transform of $f$ defined by
\[\wt{f}(\tau,n)=\frac{1}{2\pi}\int_{\R}\int_{0}^{2\pi} e^{-ixn}e^{-it\tau}f(t,x)\; dxdt .\]
Then, it is known that the (space-time) inverse Fourier transform is naturally defined as
\[f(t,x)=\frac{1}{2\pi}\sum_{n\in\Z}\int_{\R} e^{ixn}e^{it\tau}\wt{f}(\tau,n)\;dt .\]
Moreover, we use $\ft_x$ (or $\wh{\;}$ ) and $\ft_t$ to denote the spatial and temporal Fourier transform, respectively.

\medskip

For $\lambda \ge 1$, we define rescaled periodic domain $\T_{\lambda}$ and its Fourier domain $\Z_{\lambda}$ by
\[\T_{\lambda}:= \R / (2\pi \lambda \Z) = [0,2\pi\lambda] \quad \mbox{and} \quad  \Z_{\lambda}:= \left\{\frac{m}{\lambda} : m \in \Z\right\}.\]
Then, the following Fourier transform and its properties are well-known (see, for instance, \cite[Section 7]{CKSTT2003}):
\begin{itemize}
\item For a function $f$ defined on $\T_{\lambda}$, set
\[\int_{\T_{\lambda}} f(x) \; dx := \int_0^{2\pi\lambda} f(x) \; dx\]
\item For a function $f$ defined on $\Z_{\lambda}$, set
\[\int_{\Z_{\lambda}} f(n) \; dn := \frac{1}{2\pi\lambda}\sum_{n \in \Z_{\lambda}} f(n) .\]
\item $\ell^p(\Z_{\lambda})$ norm for a function $f$ defined on $\Z_{\lambda}$
\[\norm{f}_{\ell^p(\Z_{\lambda})} :=\left( \int_{\Z_{\lambda}} |f(n)|^p \; dn\right)^{\frac1p}, \quad 1 \le p < \infty, \quad \mbox{and} \quad \norm{f}_{\ell^{\infty}(\Z_{\lambda})} := \sup_{n \in \Z_{\lambda}} |f(n)|.\]
\item Fourier transform for a function $f$ defined on $\T_{\lambda}$
\[\wh{f}(n) := \int_{\T_{\lambda}} e^{-inx} f(x) \; dx, \quad n \in \Z_{\lambda}.\]
\item Fourier inversion formula for a function $f$ defined on $\T_{\lambda}$
\[f(x) := \int_{\Z_{\lambda}} e^{inx} \wh{f}(n) \; dx, \quad x \in \T_{\lambda}.\]
\item Parseval identity
\[\int_{\T_{\lambda}} f(x) \overline{g(x)} \; dx = \int_{\Z_{\lambda}} \wh{f}(n) \overline{\wh{g}(n)} \; dn.\]
\item Plancherel theorem
\[\norm{f}_{L^2(T_{\lambda})} = \norm{\wh{f}}_{\ell^2(\Z_{\lambda})}.\]
\item Convolution
\[\wh{fg}(n) =\int_{\Z_{\lambda}}\wh{f}(n-m)\wh{g}(m)\; dm.\]
\item Derivatives
\[\partial_x^j f (x) = \int_{\Z_{\lambda}}e^{inx}(in)^j \wh{f}(n) \; dn, \quad j \in \Z_{\ge 0}.\]
\item Sobolev space
\[\norm{f}_{H^s(\T_{\lambda})} = \norm{\bra{n}^s \wh{f}}_{\ell^2(\Z_{\lambda})},\]
where $\bra{\cdot} = (1+|\cdot|^2)^{\frac12}$. 
\end{itemize}
We denote $L^p(\T_{\lambda})$, $\ell^p(\Z_{\lambda})$ and $H^s(\T_{\lambda})$ by $L_{\lambda}^p$, $\ell_{\lambda}^p$ and $H_{\lambda}^s$, respectively, unless there is a risk of confusion.
\subsection{Setting}
Taking the Fourier transform to \eqref{eq:kawahara}, one has
\begin{equation}\label{eq:kawahara0}
\pt \wh{v}(n) - ip_0(n) \wh{v}(n) = \frac{\mu i}{3(2\pi)^2}n \sum_{n=n_1+n_2+n_3} \wh{v}(n_1)\wh{v}(n_2)\wh{v}(n_3),
\end{equation}
where 
\[p_*(n) = n^5 + \beta n^3- \gamma n.\]
Note that the resonance function generated by $p_0(n)$ (in the cubic interactions) is given by
\begin{equation}\label{eq:resonance function1}
\begin{aligned}
H_* &= H_*(n_1,n_2,n_3,n) \\
&:= p_*(n) - p_*(n_1) - p_*(n_2) - p_*(n_3)\\
&= \frac52 (n_1+n_2)(n_2+n_3)(n_3+n_1)\left(n_1^2 + n_2^2 + n_3 ^2 + n^2 + \frac65\beta \right)
\end{aligned}
\end{equation}
It is known from \eqref{eq:resonance function1} that the non-trivial resonances appear when $H_* = 0$, equivalently, $(n_1+n_2)(n_2+n_3)(n_3+n_1) = 0$. 

We split the nonlinear term in \eqref{eq:kawahara0} into two parts, and hence we rewrite \eqref{eq:kawahara0} as follows:
\begin{equation}\label{eq:kawahara0-1}
\begin{aligned}
\pt \wh{v}(n) - ip_*(n) \wh{v}(n) =&~{} -\frac{\mu i}{(2\pi)^2} n|\wh{v}(n)|^2 \wh{v}(n) +\frac{\mu i}{2\pi}n \left(\frac{1}{2\pi}\sum_{n' \in \Z} |\wh{v}(n')|^2 \right)\wh{v}(n)\\
&~{}+\frac{\mu i}{3(2\pi)^2}n \sum_{\N_n} \wh{v}(n_1)\wh{v}(n_2)\wh{v}(n_3),
\end{aligned}
\end{equation}
where $\N_n$ is the set of frequencies (with respect to the fixed frequency $n$), for which the relations of frequencies never generate the resonance, given by
\[\N_n = \set{(n_1,n_2,n_3) \in \Z^3 : n_1+n_2+n_3 = n, \; (n_1+n_2)(n_2+n_3)(n_3+n_1) \neq 0}.\]
We call the first two terms in the right-hand side of \eqref{eq:kawahara0-1} (non-trivial) \emph{resonant terms} and the rest \emph{non-resonant term}. 

\begin{remark}
Compared to the non-periodic problem, such resonant terms are one of obstacles to study the "low regularity" local theory of periodic dispersive equations, while the (exact) resonant phenomena can be never seen in the non-periodic dispersive equation, since the set of frequencies, which generate the resonance (in this case $(n_1+n_2)(n_2+n_3)(n_3+n_1) =0$), is a measure zero set.
\end{remark}

The $L^2$ conservation law \eqref{eq:L2} enables us to get rid of the second term in the resonant terms, so that we reduce \eqref{eq:kawahara0-1} by 
\begin{equation}\label{eq:kawahara1}
\pt \wh{v}(n) - ip_0(n) \wh{v}(n) = -\frac{\mu i}{(2\pi)^2} n|\wh{v}(n)|^2 \wh{v}(n) +\frac{\mu i}{3(2\pi)^2}n \sum_{\N_n} \wh{v}(n_1)\wh{v}(n_2)\wh{v}(n_3),
\end{equation}
where
\begin{equation}\label{p(n)}
p_0(n) = n^5 + \beta n^3- \left(\gamma +\frac{\mu}{2 \pi}  \norm{v_{0}}_{L^2}^2\right) n.
\end{equation}

\begin{remark}\label{rem:gauge}
One can use the Gauge transform defined by
\begin{equation}\label{eq:gauge}
\mathcal{G}[v](t) := e^{i\mu t\oint |v(t,x)|^2 \; dx}v(t),
\end{equation}
where $\oint f \; dx = \frac{1}{2\pi} \int_0^{2\pi} f \; dx$, in addition to the $L^2$ conservation law, to reduce \eqref{eq:kawahara} by the renormalized (or Wick ordered) modified Kawahara equation
\begin{equation}\label{eq:RKE}
\pt v - \px^5 v + \beta \px^3 v + \gamma \px v - \mu \left(v^2 - \oint v^2 \; dx \right)\px v = 0.
\end{equation}
It is well-known that the Gauge transform \eqref{eq:gauge} is well-defined and invertible when $s \ge 0$, thanks to $L^2$ conservation law \eqref{eq:L2}. Moreover, this reduction \eqref{eq:RKE} is identical to \eqref{eq:kawahara1} in the sense that no more dispersive smoothing effect arises in the cubic interactions. 
\end{remark}

\begin{remark}
It is clear that the resonant term $i n|\wh{v}(n)|^2 \wh{v}(n)$ has an effect on the solution in the sense that the solution oscillates rapidly so that the uniform continuity of the solution map breaks in $H^s(\T)$, $s < \frac12$. In fact, the estimate of this term is valid for $s \ge \frac12$ (see Lemma \ref{lem:resonant1} below), thus the local well-posedness of \eqref{eq:kawahara} is naturally expected to hold at this regularity.
\end{remark}

We denote the resonant and the non-resonant terms in \eqref{eq:kawahara1} by $\N_R(v)$ and $\N_{NR}(v)$, respectively, and these can be generally defined by
\begin{equation}\label{eq:resonant term}
\N_R(v_1,v_2,v_3) = \ft_x^{-1}\left[-\frac{\mu i}{(2\pi)^2} n \wh{v}_1(n)\wh{v}_2(-n)\wh{v}_3(n) \right] 
\end{equation}
and
\begin{equation}\label{eq:non-resonant term}
\N_{NR} (v_1,v_2,v_3) =  \ft_x^{-1}\left[\frac{\mu i}{3(2\pi)^2}n \sum_{\N_n} \wh{v}_1(n_1)\wh{v}_2(n_2)\wh{v}_3(n_3) \right].
\end{equation}

The modified linear operator in \eqref{eq:kawahara1} (defined by $p_0(n)$ in the Fourier mode) generates another cubic resonance function given by
\begin{equation}\label{eq:resonance function2}
\begin{aligned}
H_0 &= H_0(n_1,n_2,n_3,n) \\
&:= p_0(n) - p_0(n_1) - p_0(n_2) - p_0(n_3)\\
&= \frac52 (n_1+n_2)(n_2+n_3)(n_3+n_1)\left(n_1^2 + n_2^2 + n_3 ^2 + n^2 + \frac65\beta \right) = H_*.
\end{aligned}
\end{equation}
It is noted that the resonance function \eqref{eq:resonance function2} is identical to \eqref{eq:resonance function1}, since the first-order linear operator does not produce the dispersive effect as mentioned in Remark \ref{rem:gauge}.

\medskip

The standard Fourier restriction norm method ensures the local well-posedness of \eqref{eq:kawahara} in $H^s(\T)$, $s \ge \frac12$. It follows from the resonance and non-resonance estimates at such regularities, see Lemmas \ref{lem:resonant1} and \ref{lem:nonres-trilinear}.

\medskip

On the other hand, the main purpose of this paper (as seen in Section \ref{sec:intro}) is to show the well-posedness of \eqref{eq:kawahara} below $H^{\frac12}(\T)$. In view of Lemma \ref{lem:resonant1} (compared to Lemma \ref{lem:nonres-trilinear}), one can see that the resonant term $\N_{R}(v)$ in \eqref{eq:kawahara1} prevents the regularity threshold from going down below $\frac12$. 

\medskip

Takaoka and Tsutsumi \cite{TT2004} introduced new idea to weaken the nonlinear perturbation of the form $in|\wh{u}(n)|^2\wh{u}(n)$ in the context of modified KdV equation. We briefly explain what the idea is. The evolution operator $\mathcal{V}(t)$ given by
\[\mathcal{V}(t)v := \frac{1}{2\pi} \sum_{n \in \Z} e^{i(nx + tp_1(n) - \mu n\int_0^t |\wh{v}(s,n)|^2 \; ds)}\wh{v}(n)\]
can completely remove whole non-trivial resonance $\N_{R}(v)$ in \eqref{eq:kawahara1}, while the nonlinear oscillation factor $e^{-i\mu \int_0^t |\wh{v}(s,n)|^2 \; ds}$ itself is difficult to be dealt with, in particular, in the uniqueness part (see Theorem 1.1 in \cite{NTT2010} for the existence result), since the oscillation factor contains the solution to be estimated. Instead, by choosing the first approximation of $\mathcal{V}(t)$ given by
\[\mathcal{W}(t)v := \frac{1}{2\pi} \sum_{n \in \Z} e^{i(nx + tp_1(n) +  t n|\wh{v}_0(n)|^2)}\wh{v}(n),\]
we further reduce \eqref{eq:kawahara1} to
\begin{equation}\label{eq:kawahara la=1}
\pt \wh{v}(n) - ip(n) \wh{v}(n) = -\frac{\mu i}{(2\pi)^2} n(|\wh{v}(n)|^2 - |\wh{v}_0(n)|^2) \wh{v}(n) + \frac{\mu i}{3(2\pi)^2}n \sum_{\N_n} \wh{v}(n_1)\wh{v}(n_2)\wh{v}(n_3),
\end{equation}
where 
\begin{equation}\label{eq:RF}
p(n) = n^5 + \beta n^3-\left(\gamma +\frac{\mu}{2 \pi }  \norm{v_{0}}_{L^2}^2\right) n -\frac{\mu n}{(2\pi)^2} |\wh{v}_{0}(n)|^2.
\end{equation}
Then, the authors proved (in the context of modified KdV) a kind of smoothing effect (to control the reduced resonant term $n(|\wh{v}(n)|^2 - |\wh{v}_0(n)|^2)$), and thus showed the local well-posedness below $H^{\frac12}(\T)$.

\begin{remark}\label{rem:negligible factor}
One can immediately check that the resonance function \eqref{eq:resonance function2} is roughly bounded above by $\max(|n_1|^3, |n_2|^3, |n_3|^3, |n|^3)$\footnote{This upper bound is the weakest dispersive effect arising in the \emph{high-high-high} to \emph{high} interactions.} on the set $\N_n$. On the other hand, $|n||\wh{v}_0(n)|^2$ is much less than $|n|^3$, when $v_0 \in H^s(\T)$ for $-1<s$. These observations ensure that the new resonant function generated by $p(n)$ shows the same effect as \eqref{eq:resonance function2} in the analysis, in other words, the factor $n|\wh{v}_0(n)|^2$ is negligible compared to \eqref{eq:resonance function2} for $s > -1$.
\end{remark}

The standard nonlinear estimates and a kind of smoothing effect ensure to attain the local well-posedness in $L^2$ (see Sections \ref{sec:tri} and \ref{sec:energy}), provided that $L^4$ estimate is valid for $L^2$ data. In view of the proof of $L^4$ estimate in \cite{TT2004}, one can see that the proof basically relies on the smallness of $|\wh{v}_0(n)|^2$ \footnote{Thus, $L^4$ estimate holds true for $s > 0$ depending only on $\norm{v_0}_{H^s}$. This is important, because the local well-posedness follows from the compactness argument, so that the implicit constants in everywhere do not be required to depend on $v_0$.}, thus, the scaling argument enables us to  avoid the lack of $L^4$ estimate for $L^2$ data. It becomes another obstacle to obtain the global well-posedness of \eqref{eq:kawahara} in $L^2(\T)$, in a sharp contrast to others \cite{TT2004, NTT2010, MT2018, MPV2018, OW2018, Kwak2018-1}. 

\medskip

Let $\lambda \ge 1$. The same argument for \eqref{eq:scaled kawahara} gives
\begin{equation}\label{eq:lambda-kawahara}
\begin{aligned}
\pt \wh{v}_{\lambda}(n) - ip_{\lambda}(n) \wh{v}_{\lambda}(n) =&~{} -\frac{\mu i n}{(2\pi \lambda)^2}(|\wh{v}_{\lambda}(n)|^2 - |\wh{v}_{\lambda ,0}(n)|^2) \wh{v}_{\lambda}(n) \\
&+ \frac{\mu i}{3(2\pi \lambda)^2}n \sum_{\N_n^{\lambda}} \wh{v}_{\lambda}(n_1)\wh{v}_{\lambda}(n_2)\wh{v}_{\lambda}(n_3),
\end{aligned}
\end{equation}
where 
\[p_{\lambda}(n) =  n^5 + \beta \lambda^{-2}n^3- \left(\gamma +\frac{\mu}{2 \pi \lambda}  \norm{v_{\lambda,0}}_{L_{\lambda}^2}^2\right) n -\frac{\mu n}{(2\pi \lambda)^2} |\wh{v}_{\lambda, 0}(n)|^2,\]
and
\[\N_n^{\lambda} = \set{(n_1,n_2,n_3) \in \Z_{\lambda}^3 : n_1+n_2+n_3 = n, \; (n_1+n_2)(n_2+n_3)(n_3+n_1) \neq 0}.\]
Note that $p_{\la}(n) = p(n)$ for $p(n)$ as in \eqref{eq:RF}, when $\la = 1$. The resonant function $H_{\la}$ generated by $p_{\la}(n)$ is explicit by
\begin{equation}\label{eq:resonant function la}
\begin{aligned}
H_{\la} &= H_{\la}(n_1,n_2,n_3,n) \\
&:= p_{\la}(n) - p_{\la}(n_1) - p_{\la}(n_2) - p_{\la}(n_3)\\
&= \frac52 (n_1+n_2)(n_2+n_3)(n_3+n_1)\left(n_1^2 + n_2^2 + n_3 ^2 + n^2 + \frac65\beta\la^{-2} \right) \\
&-\frac{\mu n}{(2\pi \lambda)^2} |\wh{v}_{\lambda, 0}(n)|^2 + \sum_{j=1}^3 \frac{\mu n_j}{(2\pi \lambda)^2} |\wh{v}_{\lambda, 0}(n_j)|^2\\
=:&~{} H_{\la, 0} -\frac{\mu n}{(2\pi \lambda)^2} |\wh{v}_{\lambda, 0}(n)|^2 + \sum_{j=1}^3 \frac{\mu n_j}{(2\pi \lambda)^2} |\wh{v}_{\lambda, 0}(n_j)|^2.
\end{aligned}
\end{equation}

For the sake of simplicity, we replace $v_{\lambda}$ by $u$ for a $2\pi \la$-periodic function $u$ with $u_0 = v_{\lambda, 0}$, if there is no risk of confusion. Then, $u$ solves
\begin{equation}\label{eq:lambda-kawahara_u}
\begin{aligned}
\pt \wh{u}(n) - ip_{\lambda}(n) \wh{u}(n) =&~{} -\frac{\mu i n}{(2\pi \lambda)^2}(|\wh{u}(n)|^2 - |\wh{u}_0(n)|^2) \wh{u}(n) \\
&+ \frac{\mu i}{3(2\pi \lambda)^2}n \sum_{\N_n^{\lambda}} \wh{u}(n_1)\wh{u}(n_2)\wh{u}(n_3).
\end{aligned}
\end{equation}
We denote the non-resonant term in \eqref{eq:lambda-kawahara_u} by $\N_{NR}^{\la}(u)$, and it can be generally defined by
\[\N_{NR}^{\la} (u_1,u_2,u_3) =  \ft_x^{-1}\left[\frac{\mu i}{3(2\pi\la)^2}n \sum_{\N_n^{\la}} \wh{u}_1(n_1)\wh{u}_2(n_2)\wh{u}_3(n_3) \right].\]

We end this section with three remarks.
\begin{remark}\label{rem:negligible factor la}
Remark \ref{rem:negligible factor} is available for $\la$-scaled system \eqref{eq:lambda-kawahara}. Indeed, $\frac{n}{\la^2}|\wh{v}_{\la,0}|^2$ in the resonant function $H_{\la}$ as in \eqref{eq:resonant function la} is negligible compared to $H_{\la,0}$ for $s > -1$.
\end{remark}
 
\begin{remark}
The key in the reduction of the non-trivial resonance here, is that the (reduced) resonant term, in particular, $\frac{n}{\la^2}(|\wh{u}(n)|^2 - |\wh{u}_0(n)|^2)$, has a smoothing effect (see Corollary \ref{cor:main}). Indeed, using \eqref{eq:lambda-kawahara_u}, one has
\begin{equation}\label{eq:reduction}
\frac{n}{\la^2}(|\wh{u}(n)|^2 - |\wh{u}_0(n)|^2) = \frac{C}{\la^4} \mbox{Im} \left[ \int_0^t \sum_{\N_n^{\la}} n^2 \wh{u}(s,n_1)\wh{u}(s,n_2)\wh{u}(s,n_3)\wh{u}(s,-n) \; ds\right],
\end{equation}
for some constant $C \in R$. The smoothing effect occurs due to the highly non-resonant structure, stronger than the loss of regularities in \eqref{eq:reduction}. 
\end{remark}

\begin{remark}
The $\mathcal{F}\ell^1$-smoothing estimate loses the (logarithmic) derivative compared to the $\mathcal{F}\ell^{\infty}$-smoothing effect. In other word, in \cite{TT2004, NTT2010},  the $\mathcal{F}\ell^1$-smoothing estimate to control $n(|\wh{u}(n)|^2 - |\wh{u}_0(n)|^2)$ has been shown for $\frac13 < s < \frac12$, in the context of modified KdV equation, see also \cite{MT2018, OW2018}. In contrast with this, the $\mathcal{F}\ell^{\infty}$-smoothing effect (Corollary \ref{cor:main}) holds even in the end point regularity, see also \cite{Kwak2018-1, MPV2018}. Among other works, this observation is significant and the $\mathcal{F}\ell^{\infty}$-smoothing effect is essential in this work in the sense that we obtain the local well-posedness in $L^2$, and so global well-posedness in $L^2$. This observation may recover the lack of the well-posedness at the end point regularity in \cite{TT2004, NTT2010, MT2018, OW2018}.
\end{remark}

\begin{remark}
The scaling argument is not necessary for the local well-posedness in $H^s$, $s >0$ (see Remark \ref{rem:s>0}). However, we use it to prove the local well-posedness in $H^s$, $s < \frac12$, since the $L^2$ local well-posedness (thus, global well-posedness in $L^2$) is the principal aim in this work.
\end{remark}

\subsection{Function spaces}
We, in this section, introduce the $X^{s,b}$ space, which was first proposed by Bourgain \cite{Bourgain1993} to solve the periodic NLS and generalized KdV. Later, for three decades, many mathematicians, in particular,  Kenig, Ponce and Vega \cite{KPV1996} and Tao \cite{Tao2001}, have further developed.

\medskip 

For $\lambda \ge 1$, define the $X^{s,b}(\R \times \T_{\lambda})$ (shortly denoted by $X_{\lambda}^{s,b}$) space as the closure of Schwartz functions $\Sch_{t,x}(\R \times \T_{\lambda})$ under the norm 
\begin{equation}\label{eq:Xsb}
\norm{f}_{X_{\lambda}^{s,b}}^2 = \int_{\Z_{\lambda}}\int_{\R} \bra{n}^{2s} \bra{\tau-p_{\lambda}(n)}^{2b}|\wt{f}(\tau,n)|^2  \; d\tau\; dn,
\end{equation}
which is equivalent to the expression $\norm{e^{it p_{\lambda}(-i\px)}f(t,x)}_{H_t^bH_{\lambda}^s}$.

\medskip

For the high regularity ($s \ge \frac12$) local well-posedness, we define the $Y^{s,b}$ space (corresponding to \eqref{eq:kawahara1}) as the closure of Schwartz functions $\Sch_{t,x}(\R \times \T)$ under the norm 
\[\norm{f}_{Y^{s,b}}^2 = \frac{1}{2\pi}\sum_{n \in \Z} \int_{\R} \bra{n}^{2s} \bra{\tau-p_0(n)}^{2b}|\wt{f}(\tau,n)|^2 \; d\tau,\]
which is equivalent to the expression $\norm{e^{itp_0(-i\px)}f(t,x)}_{H_t^bH_x^s}$.

\medskip

Let $0 < T \le 1$. The time localization of $Y^{s,b}$ denoted by $Y_T^{s,b}$ is given by
\[Y_T^{s,b} = \set{f \in \mathcal{D}'((-T,T) \times \T) : \norm{f}_{Y_T^{s,b}} < \infty},\]
equipped with the norm
\[\norm{f}_{Y_T^{s,b}} = \inf \set{\norm{g}_{Y^{s,b}} : g \in Y^{s,b}, \; g \equiv f \; \mbox{on} \; (-T, T)}.\]

For a cut-off function $\psi$ given by
\[\psi \in C_0^{\infty}(\mathbb{R}) \quad \mbox{such that} \quad 0 \le \psi \le1, \quad  \psi \equiv 1 \; \mbox{ on } \; [-1,1], \quad \psi \equiv 0, \; |t| \ge 2,\]
we fix the time localized function 
\[\psi_T(t) = \psi (t/T), \qquad 0 < T < 1.\]

The following lemma provides the selective properties of $X_{\la}^{s,b}$ and $Y^{s,b}$ spaces.
\begin{lemma}[Properties of $X^{s,b}$ and $Y^{s,b}$, \cite{GTV1997, CKSTT2003, Tao2006}]\label{lem:Xsb}
Let $\lambda \ge 1$ $0 < T \le 1$, $s \in \R$ and $b > \frac12$. We have
\begin{enumerate}
\item (Embedding) For any $u \in X^{s,b}(\R \times \T_{\lambda})$,
we have
\[\norm{u}_{C_t^0H^s(\R \times \T_{\lambda})} \lesssim_b \norm{u}_{X_{\lambda}^{s,b}}.\]
Similarly, for any $u \in Y^{s,b}(\R \times \T)$\footnote{One can extend the domain $\R \times \T$ to $\R \times \mathcal{Z}$, where $\mathcal{Z} = \R^d$ or $\T^d$, $d \ge 1$.}, we have
\[\norm{u}_{C_t^0H^s(\R \times \T)} \lesssim_b \norm{u}_{Y^{s,b}}.\]
\item ($X^{s,b}$-type energy estimate)  For any functions $u \in \mathcal{S}_{t,x}(\R \times \T_{\lambda})$ satisfying 
\[\pt \wh{u}(n) - ip_{\lambda}(n)\wh{u}(n) = \wh{F}(u)(n),\]
we have
\[\norm{\psi(t)u}_{X_{\lambda}^{s,b}} \lesssim_{\psi, b} \norm{u_{0}}_{H_{\lambda}^s} + \norm{F}_{X_{\lambda}^{s,b-1}}.\]
Moreover, for $-\frac12 < b' < 0 < b < b'+1$ and any functions $u \in \mathcal{S}_{t,x}(\R \times \T)$ satisfying 
\[\pt \wh{u}(n) - ip_0(n)\wh{u}(n) = \wh{F}(u)(n),\]
\eqref{eq:kawahara}, we have
\[\norm{\psi_Tu}_{Y^{s,b}} \lesssim_{\psi, b} T^{\frac12-b}\norm{u_{0}}_{H^s} + T^{b'-b+1}\norm{F(u)}_{Y^{s,b'}}.\]
\end{enumerate}
\end{lemma}

\subsection{Basic estimates}
This section devotes to the introduction of some lemmas, which will be essentially used for our analysis. 
\begin{lemma}[$L^4$-Strichartz estimate]\label{lem:L4}
Let $\lambda \ge 1$. Assume that $\norm{v_{\lambda,0}}_{L_{\la}^2} \le \rho$, for a sufficient small $0 < \rho \ll 1$, and $b > \frac{3}{10}$. Let the norm $X_{\lambda}^{0,b}$ be defined as in \eqref{eq:Xsb} for $v_{\lambda,0}$. Then,
\[\norm{f}_{L_{t,x}^4(\R \times \T_{\lambda})} \lesssim \norm{f}_{X_{\lambda}^{0,b}}.\]
for any Schwartz function $f$ on $\R \times \T_{\lambda}$. The implicit function does not depend on $\lambda$.
\end{lemma}
\begin{proof}
The $L^4$-type estimate was first introduced by Bourgain \cite{Bourgain1993}, for which the local well-posedness of periodic NLS and gKdV equations have been proved. The $L^4$ estimate plays an important role to compensate for the lack of smoothing effect under the periodic setting. The proof is analogous to one in \cite{TT2004}. We leave the proof in Appendix \ref{sec:L4}.
\end{proof}

\begin{remark}\label{rem:s>0}
Lemma \ref{lem:L4} is valid, when $\la = 1$ and $v_{0} \in H^s$, $s > 0$, which facilitates that nonlinear estimate and the smoothing effect established in Sections \ref{sec:tri} and \ref{sec:energy} holds true for $s >0$ without smallness condition of initial data. This fact guarantees the local well-posedness in $H^s$, $s > 0$, without the scaling argument.
\end{remark}

\begin{remark}
The scaling argument ensures the smallness of $\norm{v_{\la,0}}_{L^2}$, for a sufficiently large $\la \gg 1$, whenever $v_{\la,0} \in H_{\la}^s$, $s \ge 0$, see \eqref{eq:scaling argument}. 
\end{remark}

\begin{remark}
Lemma \ref{lem:L4} is valid for $Y^{s,b}$ functions, i.e., we have 
\[\norm{f}_{L_{t,x}^4(\R \times \T)} \lesssim \norm{f}_{Y^{0,b}},\]
for any Schwartz function $f$ on $\R \times \T$ and $b \ge \frac{3}{10}$.
\end{remark}

\begin{lemma}[Sobolev embedding]\label{lem:sobolev}
Let $\la \ge 1$. Let $2 \le p < \infty$ and $f$ be a smooth function on $\R \times \T_{\la}$. Then for $b \ge \frac12 - \frac1p$, we have
\[\norm{f}_{L_t^p(H_{\la}^s)} \lesssim \norm{f}_{X_{\la}^{s,b}}.\]
When $p = \infty$, the usual Sobolev embedding ($b > \frac12$) holds.
\end{lemma}

\begin{proof}
The proof directly follows from the Sobolev embedding with respect to the temporal variable $t$. For $S_{\la}(t)f(t,x) = \ft_x^{-1}[e^{itp_{\la}(n)}\wh{f}(t,n)]$, we know $\norm{S_{\la}(-t)f}_{H_{\la}^s} = \norm{f}_{H_{\la}^s}$. Thus,
\[\norm{f}_{L_t^p(H_{\la}^s)} = \norm{\norm{S_{\la}(-t)f}_{H_{\la}^s}}_{L_t^p} \lesssim \norm{\norm{S_{\la}(-t)f}_{H_{\la}^s}}_{H_t^b}= \norm{f}_{X_{\la}^{s,b}},\]
which completes the proof.
\end{proof}

\begin{remark}\label{rem:sobolev Y}
Lemma \ref{lem:sobolev} is valid for $Y^{s,b}$ functions, i.e., we have
\[\norm{f}_{L_t^p(H^s)} \lesssim \norm{f}_{Y^{s,b}}.\]
\end{remark}

\section{Trilinear estimates}\label{sec:tri}
In this section, we establish the trilinear estimates, which is the main task in the Fourier restriction norm method. We split the nonlinear estimate into two: Resonance and Non-resonance estimates.
\begin{lemma}[Resonance estimate]\label{lem:resonant1}
Let $s \ge \frac12$, $\frac13 \le b < 1$ and $0 < \delta \le 1-b$. For $N_{R}$ as in \eqref{eq:resonant term}, we have
\begin{equation}\label{eq:resonant1}
\norm{\N_R(u_1,u_2,u_3)}_{Y_T^{s,b-1+\delta}} \lesssim \prod_{j=1}^{3} \norm{u_j}_{Y_T^{s,b}}.
\end{equation}
\end{lemma}

\begin{proof}
For the term $n\wh{u_1}(n)\wh{u_2}(-n)\wh{u_3}(n)$, since 
\[\norm{\bra{n}^{s}n\wh{u_1}(n)\wh{u_2}(-n)\wh{u_3}(n)}_{\ell_n^2} \lesssim \prod_{j=1}^3\norm{u_j}_{H^s},\]
for $s \ge \frac12$, we have from the Sobolev embedding (see Remark \ref{rem:sobolev Y}) that
\[\begin{aligned}
\norm{\ft_x^{-1}[in\wh{u_1}(n)\wh{u_2}(-n)\wh{u_3}(n)]}_{Y_T^{s,b-1+\delta}} &\lesssim \norm{\ft_x^{-1}[\bra{n}^{s}\wh{u_1}(n)\wh{u_2}(-n)\wh{u_3}(n)]}_{Y_T^{0,0}}\\
&\lesssim \prod_{j=1}^3 \norm{u_j}_{L_t^6(H^s)} \lesssim \prod_{j=1}^3\norm{u}_{Y_T^{s,\frac13}},
\end{aligned}\]
which implies \eqref{eq:resonant1}, we, thus, complete the proof.
\end{proof}

\medskip

In contrast to the resonance estimate, one can use the dispersive smoothing effect arising from the cubic non-resonant interactions. From the symmetry
\[n_1 + n_2 + n_3 = n \quad \mbox{and} \quad \tau_1 + \tau_2 + \tau_3 = \tau,\]
we know
\[(\tau_1 - p_{\la}(n_1)) + (\tau_2 - p_{\la}(n_2)) + (\tau_3 - p_{\la}(n_3)) = \tau - p_{\la}(n) + H_{\la}.\]
Thus, we alway assume that
\[\max(|\tau_j - p(n_j)|, |\tau - p(n)| ; j= 1,2,3) \gtrsim |H_{\la}| \sim |H_{\la, 0}|,\]
whenever $s > -1$.

\begin{lemma}[Non-resonance estimate]\label{lem:nonres-trilinear}
Let $\lambda \ge 1$. Assume that $\norm{v_{\lambda,0}}_{L^2} \le \rho$, for a sufficiently small $0 < \rho \ll 1$. Then, for $s > -\frac14$ there exist $0 < \delta = \delta(s) \ll 1$ and $\frac12 \le b = b(s)$ such that the following estimate holds:
\begin{equation}\label{eq:NRE}
\norm{\N_{NR}^{\la}(u_1,u_2,u_3)}_{X_{\la}^{s,b-1+\delta}} \lesssim \la^{\varsigma}\prod_{j=1}^3 \norm{u_j}_{X_{\la}^{s,b}},
\end{equation}
for $\varsigma = \varsigma(s,b,\delta) = 3-5b-10\delta +2s > 0$.
\end{lemma}

\begin{remark}\label{rem:non-resonance estimate for s>1/2}
Lemma \ref{lem:nonres-trilinear} is valid in $Y_T^{s,b}$ without the constraint on the size of $L^2$ norm of initial data, whenever $s >-\frac14$. Precisely, for $-\frac14 \le s $ and $0 \le T \le 1 $, there exist $0 < \delta = \delta(s) \ll 1$ and $\frac12 \le b = b(s)$ such that the following estimate holds:
\begin{equation}\label{eq:NRE 12}
\norm{\N_{NR}(u_1,u_2,u_3)}_{Y_T^{s,b-1+\delta}} \lesssim \prod_{j=1}^3 \norm{u_j}_{Y_T^{s,b}}.
\end{equation}
Together with \eqref{eq:resonant1} and \eqref{eq:NRE 12}, we prove the local well-posedness in $H^s$, $s \ge \frac12$. See Section \ref{sec:WP 12} for more details.
\end{remark}

\begin{proof}[Proof of Lemma \ref{lem:nonres-trilinear}]
From the duality argument in addition to the Plancherel theorem, the left-hand side of \eqref{eq:NRE} is equivalent to
\begin{equation}\label{eq:Planchere}
C\int_{\Z_{\la}}\int_{\R} \frac{1}{\la^2} \sum_{\N_n^{\la}} \iint_{\R^2}\wt{u}_1(\tau_1,n_1)\wt{u}_2(\tau_2,n_2)\wt{u}_3(\tau_3,n_3)\wt{g}(-\tau, -n)\; d\Gamma_{\tau} \; d\tau \; dn,
\end{equation}
for some constant $C$, where $\int_{\R^2} \cdot \; d\Gamma_{\tau}$ means two dimensional integral over the hyper-surface $\{(\tau_1,\tau_2,\tau_3) \in \R^3 : \tau_1 + \tau_2 +\tau_3 =\tau\}$ and $g \in X_{\la}^{-s,1-b-\delta}$ with $\norm{g}_{X_{\la}^{-s,1-b-\delta}} \le 1$. Then, it suffices to show
\begin{equation}\label{eq:MLE0}
\int_\R \frac{1}{\la^3}\sum_{n \in \Z_{\la}, \N_n^{\la}} \M_0 \wh{f}_1(n_1)\wh{f}_2(n_2)\wh{f}_3(n_3)\wh{h}(-n) \; dt  \lesssim \la^{0+}\prod_{j=1}^3\norm{u_j}_{X_{\la}^{s,b}}\norm{h}_{X_{\la}^{0,1-b-\delta}},
\end{equation}
where $\M_0 = \M_0(n_1,n_2,n_3,n)$ is a Fourier multiplier defined as
\[\M_0(n_1,n_2,n_3,n) = |n|\bra{n}^s\bra{n_1}^{-s}\bra{n_2}^{-s}\bra{n_3}^{-s},\]
$n_{\ast} = \max(|n_1|, |n_2|, |n_3|, |n|)$, $f_j = \mathcal F^{-1}(\bra{n}^s|\wt{u}_j(\tau, n)|)$ and $h = \mathcal F^{-1}(\bra{n}^{-s}|\wt{g}(\tau, n)|)$.

\medskip

\begin{remark}\label{rem:generality}
Let $L_j = \bra{\tau_j - p_{\la}(n_j)}$ and $L = \bra{\tau - p_{\la}(n)}$.  When $L \ge \max(L_1,L_2,L_3)$, thanks to the Sobolev embedding (Lemma \ref{lem:sobolev}) and $L^4$ Strichartz estimate (Lemma \ref{lem:L4}), it suffices for \eqref{eq:MLE0} to show 
\begin{equation}\label{eq:MLE}
\int \frac{1}{\la^3}\sum_{n \in \Z_{\la}, \N_n^{\la}} \M \wh{f}_1(n_1)\wh{f}_2(n_2)\wh{f}_3(n_3)\wh{f}_4(-n) \; dt  \lesssim \la^{0+}\prod_{j=1}^2\norm{f_j}_{L_{t,\la}^4}\norm{f_3}_{L_t^{\infty}L_{\la}^2}\norm{f_4}_{L_{t,{\la}}^2},
\end{equation}
where the multiplier $\M$ is given by
\begin{equation}\label{eq:FM}
\M(n_1,n_2,n_3,n) = \frac{|n|\bra{n}^s\bra{n_1}^{-s}\bra{n_2}^{-s}\bra{n_3}^{-s}}{(|n_1+n_2||n_2+n_3||n_3+n_1| (n_{\ast})^2)^{1-b-2\delta}}
\end{equation}
and $f_4 = \mathcal F^{-1}(\bra{n}^{-s}\bra{\tau - p_{\la}(n)}^{1-b-\delta}|\wt{g}(\tau, n)|)$. Here, $f_3$ instead denotes $f_3= \mathcal F^{-1}(\bra{n}^s\bra{\tau - p_{\la}(n)}^{-\delta}|\wt{u}_3(\tau, n)|)$ to make sure that the Sobolev embedding is available for $u_3$ in $X_{\la}^{s,\frac12}$ (i.e., $\norm{f_3}_{L_t^{\infty}L_{\la}^2} \lesssim \norm{u}_{X_{\la}^{s,\frac12}}$). One can change the role of $f_3$ into either $f_1$ or $f_2$ without loss of the \emph{dispersive smoothing effect} ($\bra{\tau_j - p_{\la}(n_j)}$). On the other hand, when $\max(L_1, L_2, L_3) \gg L$,  \eqref{eq:MLE0} follows from
\begin{equation}\label{eq:MLE1}
\int \frac{1}{\la^3}\sum_{n \in \Z_{\la}, \N_n^{\la}} \M \wh{f}_1(n_1)\wh{f}_2(n_2)\wh{f}_3(n_3)\wh{f}_4(-n) \; dt \lesssim \la^{0+}\prod_{j=1}^2\norm{f_j}_{L_{t,\la}^4}\norm{f_3}_{L_{t,\la}^2}\norm{f_4}_{L_t^{\infty}L_{\la}^2},
\end{equation}
if $L_3 = \max(L_1,L_2,L_3)$, where, in this case, $f_3 = \mathcal F^{-1}(\bra{n}^s\bra{\tau - p_{\la}(n)}^{b}|\wt{u}_3|$, $f_4 = \mathcal F^{-1}(\bra{\tau - p_{\la}(n)}^{1-2b-\delta}|\wt{g}(\tau,n)|)$ and $\M$ is defined as in \eqref{eq:FM}. The multiplier $\M$ is still valid due to
\[L_3^{-b}L^{-1+2b+\delta} \ll L_3^{-1+b+\delta} \lesssim |H_{\la,0}|^{-1+b+\delta}.\]
Otherwise, the exact same computation gives
\begin{equation}\label{eq:MLE2}
 \int  \frac{1}{\la^3}\sum_{n \in \Z_{\la}, \N_n^{\la}} \M \wh{f}_1(n_1)\wh{f}_2(n_2)\wh{f}_3(n_3)\wh{f}_4(-n) \; dt \lesssim \la^{0+} \norm{f_1}_{L_t^{\infty}L_{\la}^2}\norm{f_2}_{L_{t,\la}^2}\prod_{j=3}^4\norm{f_j}_{L_{t,\la}^4},
\end{equation}
where, in this case ($L_2 = \max(L_1,L_2,L_3)$), $f_1 = \mathcal F^{-1}(\bra{n}^s\bra{\tau - p_{\la}(n)}^{-\delta}|\wt{u}_1|$ and $f_2 = \mathcal F^{-1}(\bra{n}^s\bra{\tau - p_{\la}(n)}^{b-\delta}|\wt{u}_2|)$, since
\[L_2^{-b} \lesssim  L_1^{-\delta}|H_{\la,0}|^{-b+\delta} \lesssim  L_1^{-\delta}|H_{\la,0}|^{-1+b+2\delta}.\] 
One can switch the roles between $f_1$ and $f_2$, if $L_1 = \max(L_1,L_2,L_3)$. In view of the proof below, no more assumption is needed for \eqref{eq:MLE}, \eqref{eq:MLE1} and \eqref{eq:MLE2}, and thus it suffices to show \eqref{eq:MLE} with $\M$ as in \eqref{eq:FM}.

\medskip

Finally we notify that the multiplier $\M$ as in \eqref{eq:FM} will be replaced by appropriate bound for each case in the proof below. It could be achieved when reducing \eqref{eq:MLE} from \eqref{eq:Planchere}.
\end{remark}

\medskip

\begin{remark}
When $n_* \le 1$, we know $\M_0 \le 1$ in the left-hand side of \eqref{eq:MLE0}. Thus, the change of variable ($n' = n_1 + n_2$) yields
\[\begin{aligned}
\frac{1}{\la^3}\sum_{n \in \Z_{\la}, \N_n^{\la}} =&~{} \frac{1}{\la}\sum_{|n'| \ge \la^{-1}} \wh{f_1f_2}(n')\wh{f_3f_4}(-n')\\
\lesssim&~{}\norm{f_1f_2}_{L_{\la}^2}\norm{f_3f_4}_{L_{\la}^2}\\
\lesssim&~{}\prod_{j=1}^4\norm{f_j}_{L_{\la}^4}, 
\end{aligned}\]
which implies 
\[\mbox{LHS of } \eqref{eq:MLE0} \lesssim \prod_{j=1}^3\norm{u_j}_{X_{\la}^{s,b}}\norm{h}_{X_{\la}^{0,1-b-\delta}},\]
whenever $\frac{3}{10} < 1-b-\delta$. Thus, we may assume that $n_* > 1$ without loss of generality.
\end{remark}

\medskip

\textbf{Case I} (high $\times$ high $\times$ high $\Rightarrow$ high, $|n_1|\sim|n_2|\sim|n_3|\sim|n|\sim n_{\ast}$). We may assume that $\la^{-1} < |n_1+n_2| \ll n_{\ast}$ without loss of generality, since otherwise, \eqref{eq:MLE}, \eqref{eq:MLE1} and \eqref{eq:MLE2} can be obtained with different fairs of functions in $L_{t,\la}^4$ and $L_t^{\infty}L_{\la}^2$-$L_{t,\la}^2$ norms, i.e. $(f_2,f_3)$ in $L_{t,\la}^4$ and $(f_1,f_4)$ in $L_t^{\infty}L_{\la}^2$-$L_{t,\la}^2$ or vice versa. In this case, the multiplier $\M$ is bounded by
\[\frac{\la^{1-b-2\delta}}{|n_1+n_2|^{1-b-2\delta}n_{\ast}^{2-3b-6\delta + 2s}}.\] 
For $s > -\frac14$, we choose $\delta = \frac{4s+1}{24}$. Then, for all $\frac12 < b < \frac{2+2s-6\delta}{3}$, we have
\begin{equation}\label{eq:multiplier1}
\M \le \frac{\la^{1-b-2\delta}}{|n_1+n_2|^{3-4b-8\delta + 2s}}.
\end{equation}
We replace the original $\M$ by the right-hand side of \eqref{eq:multiplier1}. The change of variable ($n' = n_1 + n_2$) and the summation over $n_2, n_3$ yield
\[\begin{aligned}
\mbox{LHS of } \eqref{eq:MLE} &= \int \frac{1}{\la^3}\sum_{\substack{n',n_2,n_3 \in \Z_{\la} \\ |n'| \ge \la^{-1}}} \frac{\la^{1-b-2\delta}}{|n'|^{3-4b-8\delta + 2s}} \wh{f}_1(n'-n_2)\wh{f}_2(n_2)\wh{f}_3(n_3)\wh{f}_4(-n'-n_3)\\
&= \int \frac{1}{\la}\sum_{|n'| \ge \la^{-1}} \frac{\la^{1-b-2\delta}}{|n'|^{3-4b-8\delta + 2s}}\wh{f_1f_2}(n')\wh{f_3f_4}(-n').
\end{aligned}\]
For $s > -\frac14$, choose $\delta = \min(\frac{4s+1}{20}, \frac{1}{20})$. Then, taking the $\ell^{\infty}$-norm at $\wh{f_3f_4}(-n')$ and the Cauchy-Schwarz inequality for the rest, one has
\[\begin{aligned}
\mbox{LHS of } \eqref{eq:MLE} &\lesssim \la^{3-5b-10\delta+2s} \int \norm{f_1f_2}_{L_{\la}^2}\norm{f_3f_4}_{L_{\la}^1}\\
&\lesssim \la^{3-5b-10\delta+2s}\prod_{j=1}^2\norm{f_j}_{L_{t,\la}^4}\norm{f_3}_{L_t^{\infty}L_{\la}^2}\norm{f_4}_{L_{t,\la}^2},
\end{aligned}\]
due to 
\[\frac{1}{\la}\sum_{|n'| \ge \la^{-1}} \frac{1}{|n'|^{2(3-4b-8\delta + 2s)}} < \la^{-2 + 2(3-4b-8\delta + 2s)}\]
for all $\frac12 \le b  < \frac58 + - 2\delta + \min(\frac{s}{2},0)$.

\medskip

The argument used in \textbf{Case I} could be applicable to the other cases. In what follows, we only point the bound of the multiplier $\M$ out, but omit the computation.

\medskip

\textbf{Case II} (high $\times$ high $\times$ low $\Rightarrow$ high, $|n_1| \ll |n_2| \sim |n_3| \sim |n| \sim n_{\ast}$). The choice of the minimum frequency $n_1$ does not lose of the generality in the proof below. In this case, $\M$ is bounded by
\[\frac{1}{|n_1+n_2|^{5(1-b-2\delta)-1+s}},\]
for $s \ge 0$, and 
\[\frac{1}{|n_1+n_2|^{5(1-b-2\delta)-1+2s}},\]
for $s < 0$, due to $|n_1+n_2| \sim n_{\ast}$. 

\medskip

The same computation used in \textbf{Case I} is available, when $5(1-b-2\delta)-1+s$ for $ s \ge 0$ and $5(1-b-2\delta)-1+2s > \frac12$ for $s < 0$. Thus, for $s > -\frac12$, we choose $\delta = \min(\frac{2s+1}{20}, \frac1{20})$ such that \eqref{eq:MLE} holds true for $\frac12 \le b < \frac{7}{10} -2\delta + \min(\frac{2s}{5}, 0)$.

\medskip

\textbf{Case III} (high $\times$ high $\times$high $\Rightarrow$ low, $|n| \ll |n_1|\sim |n_2| \sim |n_3| \sim n_{\ast}$). In this case, $\M$ is bounded by
\[\frac{1}{|n_1+n_2|^{5(1-b-2\delta)-1+2s}},\]
for $s \ge 0$, and 
\[\frac{1}{|n_1+n_2|^{5(1-b-2\delta)-1+3s}},\]
for $s < 0$, due to $|n_1+n_2| \sim n_{\ast}$. For $s > -\frac13$, choosing $\delta = \min(\frac{3s+1}{20}, \frac1{20})$, one shows \eqref{eq:MLE} for $\frac12 \le b < \frac{7}{10} -2\delta + \min(\frac{3s}{5}, 0)$ via the same computation used in \textbf{Case I}.

\medskip

\textbf{Case IV} (high $\times$ low $\times$ low $\Rightarrow$ high, $|n_1|, |n_2| \ll |n_3|\sim |n| \sim n_{\ast}$). The choice of the maximum frequency $|n_3|$ is to ensure $0 < |n_1+n_2| \ll n_{\ast}$, and it does not lose the generality, thanks to the same reason in \textbf{Case I}, where $0 < |n_1+n_2| \ll n_{\ast}$ is to be supposed. In this case, $\M$ is bounded by
\[\frac{1}{|n_1+n_2|^{1-b-2\delta}n_{\ast}^{4(1-b-2\delta)-1}},\]
for $s \ge 0$, and 
\[
\frac{1}{|n_1+n_2|^{1-b-2\delta}n_{\ast}^{4(1-b-2\delta)-1+2s}},
\]
for $s < 0$. For $s > -\frac12$, by taking $\delta = \min(\frac{1+2s}{20}, \frac{1}{20})$, we have \eqref{eq:MLE} for all $\frac12 \le b < \frac{7}{10} - 2\delta + \min(\frac{2s}{5},0)$.

\medskip

\textbf{Case V} (high $\times$ high $\times$ low $\Rightarrow$ low, $|n|, |n_3| \ll |n_1|\sim |n_2| \sim n_{\ast}$). Similarly as in \textbf{Case IV}, the choice of the minimum frequency $|n_3|$ is to ensure $0 < |n_1+n_2| \ll n_{\ast}$, and hence it does not lose the generality. In this case, $\M$ is bounded by
\[\frac{1}{|n_1+n_2|^{1-b-2\delta}n_{\ast}^{4(1-b-2\delta)-1+s}},\]
for $s \ge 0$, and 
\[\frac{1}{|n_1+n_2|^{1-b-2\delta}n_{\ast}^{4(1-b-2\delta)-1+3s}},\]
for $s < 0$. For $s > -\frac13$, by taking $\delta = \min(\frac{3s+1}{20}, \frac{1}{20})$, we have \eqref{eq:MLE} for all $\frac12 \le b < \frac{7}{10} - 2\delta + \min(\frac{3s}{5},0)$.

\medskip

Gathering all, for $s > -\frac14$, we can choose $0 < \delta = \min(\frac{1+4s}{20},\frac{1}{20})$ such that \eqref{eq:MLE0} holds true for all $\frac12 \le b  < \frac58 + - 2\delta + \min(\frac{s}{2},0)$.

\medskip

Note that the worst $\la$-bound in all cases is $\la^{3-5b-10\delta +2s}$.
\end{proof}

Using \eqref{eq:lambda-kawahara_u}, we have
\begin{equation}\label{eq:energy}
\pt |\wh{u}(t,n)|^2 = \frac{2\mu }{3(2\pi \lambda)^2} n\mbox{Im}\left[ \sum_{\N_n^{\lambda}} \wh{u}(n_1)\wh{u}(n_2)\wh{u}(n_3)\wh{u}(-n)\right].
\end{equation}
An immediate corollary of Lemma \ref{lem:nonres-trilinear}, in addition to \eqref{eq:energy}, is the following:

\begin{corollary}\label{cor:trilinear}
Let $\lambda \ge 1$ and $-1/4 <s$. Assume that $\norm{v_{\lambda,0}}_{L^2} \le \rho$, for a sufficiently small $0 < \rho \ll 1$. Suppose that $u$ is a real-valued smooth solution to \eqref{eq:lambda-kawahara_u} and $u \in X_{\la}^{s,b}$. Then, there exists $\frac12 \le b=b(s)$ such that the following is holds true.
\begin{equation}\label{eq:trilinear1}
\norm{u}_{H_{\la}^s}^2 \sim \frac{1}{\la}\sum_{n \in\Z_{\la}} \bra{n}^{2s}|\wh{u}(t,n)|^2\lesssim \norm{u_0}_{H_{\la}^{s}}^2 + \la^{\varsigma}\norm{u}_{X_{\la}^{s,b}}^4,
\end{equation}
for $\varsigma = \varsigma(s,b) = 3-5b+2s > 0$.
\end{corollary}

\section{Smoothing effect}\label{sec:energy}

\subsection{A priori bound}

\begin{proposition}\label{prop:main}
Let $\lambda \ge 1$. Assume that $\norm{v_{\lambda,0}}_{L_{\la}^2} \le \rho$, for a sufficient small $0 < \rho \ll 1$. Let $0 \le s < \frac12$, $t \in [0,1]$ and $u_0 \in C^{\infty}(\T_{\la})$. Suppose that $u$ is a real-valued smooth solution to \eqref{eq:lambda-kawahara_u} and $u \in X_{\la}^{s,\frac12}$. Then the following estimate holds:
\begin{equation}\label{eq:main}
\begin{aligned}
\sup_{n \in \Z_{\lambda}}& \Bigg|\frac{1}{\lambda^4}\mbox{Im}\Bigg[ \int_0^t  n^2 \sum_{\N_n^{\la}}\wh{u}(s,n_1)\wh{u}(s,n_2)\wh{u}(s,n_3)\wh{u}(s,-n) \; ds \Bigg]\Bigg|\\
&\lesssim \la^{-1}\log \la\left(\norm{u_0}_{H_{\la}^s}^4 + \left(\norm{u_0}_{H_{\la}^{s}}^2 + \la^{\frac12+2s} \norm{u}_{X_{\la}^{s,\frac12}}^4\right)^2\right) +  \la^{-\frac12}\norm{u}_{X_{\la}^{s,\frac12}}^4 + \la^{-\frac32}\norm{u}_{X_{\la}^{s,\frac12}}^6.
\end{aligned}
\end{equation}
Note that $\log \la$ is replaced by $1$ when $\la =1$.
\end{proposition}

Before proving Proposition \ref{prop:main}, we introduce the projection operator $P_N$ as follows: For $N = 2^k$, $k \in \Z_{\ge 0}$, let 
\[I_1 = [-1,1] \quad \mbox{and} \quad I_N = [-2N,-N/2] \cup [N/2, 2N], \quad N \ge 2.\]
Note that $|I_N| \sim \la N$. For a characteristic function $\chi_E$ on a set $E$, define $P_N$ by
\[\ft_x[P_Nf](n) = \chi_{I_N}(n)\wh{f}(n).\]
We use the convention
\[P_{\le N} = \sum_{M \le N} P_M, \qquad P_{>N} = \sum_{M > N} P_M.\]
Let $N \ge 1$ be given. We decompose a function $f$ into the following three pieces:
\[f = f_{low} + f_{med} + f_{high},\]
where $f_{med} =  P_{N}f$, $f_{low} = P_{\le N}f -  f_{med}$ and $f_{high} =  P_{\ge N}f - f_{med}$.

\begin{proof}[Proof of Proposition \ref{prop:main}]
The left-hand side of \eqref{eq:main} bounded by
\begin{equation}\label{eq:main1}
\sup_{N\ge1}\sup_{n \in I_N}\frac{1}{\la^4} \Bigg|\mbox{Im}\Bigg[ \int_0^t  n^2\sum_{\N_n^{\la}}\wh{u}(s,n_1)\wh{u}(s,n_2)\wh{u}(s,n_3)\wh{u}(s,-n) \; ds \Bigg]\Bigg|.
\end{equation}
We deal with \eqref{eq:main1} by dividing into several cases.

\medskip

\begin{remark}
When $N = 1$, by the Cauchy-Schwarz inequality, $L^4$  estimate and the Sobolev embedding,  we have 
\[\eqref{eq:main1} \lesssim \la^{-\frac32} \int_0^t \norm{u(s)}_{L_{\la}^2}^2\norm{u(s)}_{L_{\la}^4}^2 \; ds \lesssim \la^{-\frac32} \norm{u}_{X_{\la}^{0,b}}^4,\]
for $b > \frac{3}{10}$. Thus, we may assume that $N > 1$.
\end{remark}

\medskip

\textbf{Case I.} (\emph{high $\times$ high $\times$ high $\Rightarrow$ high}) We consider the following term:
\begin{equation}\label{eq:hhhh0}
\frac{1}{\la^4}\sup_{N>1}\sup_{n \in I_N} \Bigg|\mbox{Im}\Bigg[ \int_0^t  n^2 \sum_{\N_n^{\la}} \wh{u}_{med}(s,n_1)\wh{u}_{med}(s,n_2)\wh{u}_{med}(s,n_3)\wh{u}(s,-n) \; ds \Bigg]\Bigg|.
\end{equation}
In order to control \eqref{eq:hhhh0} in $L^2$ regularity level, it is required to use the \emph{Normal form reduction method}.

\medskip

Observe that\footnote{It is immediately known from \eqref{p(n)} that $p(n)$ is the odd function, i.e. $p(-n) = -p(n)$.}
\begin{equation}\label{eq:time derivative}
\begin{aligned}
\pt\wh{\varphi} :=&~{} \pt \left( e^{-itp_{\la}(n)}\wh{u}(n) \right) \\
=&~{} e^{-itp_{\la}(n)} \left( \pt \wh{u}(n) - ip_{\la}(n) \wh{u}(n) \right)\\
=&~{} e^{-itp_{\la}(n)} \Bigg( -\frac{\mu i n}{(2\pi \lambda)^2} (|\wh{u}(n)|^2-|\wh{u}_0(n)|^2)\wh{u}(n) +\frac{\mu i}{3(2\pi \lambda)^2}n \sum_{\N_n^{\lambda}}\wh{u}(n_1)\wh{u}(n_2)\wh{u}(n_3) \Bigg).
\end{aligned}
\end{equation}

Taking the integration by parts with respect to the time variable $s$, one has
\[\begin{aligned}
&\int_0^t  n^2\sum_{\N_n^{\la}} \wh{u}(s,n_1)\wh{u}(s,n_2)\wh{u}(s,n_3)\wh{u}(s,-n) \; ds \\
=&~{}	\int_0^t n^2 \sum_{\N_n^{\la}} e^{-is(H_{\la})}\wh{\varphi}(s,n_1)\wh{\varphi}(s,n_2)\wh{\varphi}(s,n_3)\wh{\varphi}(s,-n) \; ds\\
=&~{}-\sum_{\N_n^{\la}} \frac{n^2}{i(H_{\la})} \wh{u}(t,n_1)\wh{u}(t,n_2)\wh{u}(t,n_3)\wh{u}(t,-n)\\
&~{}+\sum_{\N_n^{\la}} \frac{n^2}{i(H_{\la})}\wh{u}_0(n_1)\wh{u}_0(n_2)\wh{u}_0(n_3)\wh{u}_0(-n) \\
&~{}+\int_0^t \sum_{\N_n^{\la}} \frac{n^2e^{-is(H_{\la})}}{i(H_{\la})} \cdot \frac{d}{ds}\left(\wh{\varphi}(s,n_1)\wh{\varphi}(s,n_2)\wh{\varphi}(s,n_3)\wh{\varphi}(s,-n) \right)\; ds,
\end{aligned}\]
where $H_{\la}$ is defined as in \eqref{eq:resonant function la}. Then, \eqref{eq:hhhh0} is reduced as follows:
\begin{equation}\label{eq:normal form}
\begin{aligned}
\eqref{eq:hhhh0} \le&~{} \frac{1}{\la^4}\sup_{N>1}\sup_{n \in I_N}  \sum_{\N_n^{\la}} \frac{n^2}{|H_{\la}|}\left| \wh{u}_{med}(t,n_1)\wh{u}_{med}(t,n_2)\wh{u}_{med}(t,n_3)\wh{u}(t,-n) \right|\\
&~{} + \frac{1}{\la^4}\sup_{N>1}\sup_{n \in I_N}  \sum_{\N_n^{\la}}\frac{n^2}{|H_{\la}|} \left|\wh{u}_{0,med}(n_1)\wh{u}_{0,med}(n_2)\wh{u}_{0,med}(n_3)\wh{u}_{0}(-n)\right|\\
&~{}+\frac{1}{\la^4}\sup_{N>1}\sup_{n \in I_N} \Bigg|\int_0^t \sum_{\N_n^{\la}} \frac{n^2e^{-is(H_{\la})}}{i(H_{\la})}\\
& \hspace{9em} \times \frac{d}{ds}\left(\wh{\varphi}_{med}(s,n_1)\wh{\varphi}_{med}(s,n_2)\wh{\varphi}_{med}(s,n_3)\wh{\varphi}(s,-n) \right) ds \Bigg|\\
=:&~{} \Sigma_1 + \Sigma_2 + \Sigma_3.
\end{aligned}
\end{equation}
From Remark \ref{rem:negligible factor la}, we know that $|H_{\la}| \sim |H_{\la,0}|$.

\medskip

For $\Sigma_1$ and $\Sigma_2$, it suffices to show
\begin{equation}\label{eq:hhhh1}
\frac{1}{\la^4}\sup_{N>1}\sup_{n \in I_N} \left| \sum_{\N_{n_4}^{\la}} \frac{n^2}{|H_{\la,0}|} \wh{f}_{1,med}(n_1)\wh{f}_{2,med}(n_2)\wh{f}_{3,med}(n_2)\wh{f}_{4,med}(-n_4) \right| \lesssim \prod_{j=1}^{4}\norm{f_j}_{H^s},
\end{equation}
for a certain regularity $s \in \R$. In this case, one can replace $\sup_{n \in I_N}$ by $\sum_{n \in I_N}$.

\medskip

Put $\wh{g}_i(n) = \bra{n}^s|\wh{f}_{i,med}(n)|$, $i =1,2,3,4$. We know from \eqref{eq:resonant function la} with Remark \ref{rem:negligible factor la} that $|H_{\la}| \gtrsim |(n_1+n_2)(n_2+n_3)(n_3+n_1)|n_4^2 $. We may assume $|n_3+n_1| \sim N$. Since $|H_{\la}|^{-1} \lesssim \la N^{-3}|n_1+n_2|^{-1}$, a straightforward calculation (after the change of variable $n'=n_1+n_2 \neq 0$) yields
\begin{equation}\label{eq:nonresonant estimate}
\begin{aligned}
\eqref{eq:hhhh1} \lesssim&~{} \frac{1}{\la^43}\sup_{N>1}N^{-(1+4s)}\sum_{\substack{n_1,n_4\\ |n'| \ge \la^{-1}}} \frac{1}{|n'|}\wh{g}_1(n_1)\wh{g}_2(n'-n_1)\wh{g}_3(n_4-n')\wh{g}_4(-n_4)\\
\lesssim&~{}\frac{1}{\la}\sup_{N>1}N^{-(1+4s)}\sum_{\la^{-1} \le |n'| \le N} \frac{1}{|n'|}\wh{g_1g_2}(n')\wh{g_3g_4}(-n')\\
\lesssim&~{} \frac{1}{\la}\sup_{N>1}N^{-(1+4s)}(\log \la + \log N)\norm{\wh{g_1g_2}}_{\ell^{\infty}}\norm{\wh{g_3g_4}}_{\ell^{\infty}}\\
\lesssim&~{} \la^{-1}\log \la \prod_{j=1}^{4}\norm{f_j}_{H_{\la}^s},
\end{aligned}
\end{equation}
whenever $1 + 4s > 0  \Rightarrow -1/4 < s < 1/2$, which, in addition to \eqref{eq:trilinear1}, implies
\begin{equation}\label{eq:boundary estimate}
\begin{aligned}
\Sigma_1 + \Sigma_2  \lesssim&~{} \la^{-1}\log \la \left(\norm{u_0}_{H_{\la}^s}^4 + \norm{u(t)}_{H_{\la}^s}^4\right)\\
\lesssim&~{} \la^{-1}\log \la\norm{u_0}_{H_{\la}^s}^4 + \la^{-1}\log \la\left(\norm{u_0}_{H_{\la}^{s}}^2 + \la^{\frac12+2s} \norm{u}_{X_{\la}^{s,\frac12}}^4\right)^2,
\end{aligned}
\end{equation}
whenever $-1/4 < s < 1/2$.

\medskip

\begin{remark}\label{rem:choice of maximum frequency}
In view of the computation \eqref{eq:nonresonant estimate}, one can know that another choice of the assumption, $|n_1+n_2| \sim N$ or $|n_2+n_3| \sim N$, does not make a difference in the result. Hence our assumption does not lose the generality.
\end{remark}

\medskip

\begin{remark}\label{rem:boundary}
An analogous argument for the estimates of the boundary terms cannot be available in the uniqueness part, since Corollary \ref{cor:trilinear} does not hold for the difference of two solutions (due to the lack of the symmetry), that is to say, the estimate
\begin{equation}\label{eq:energy_diff}
\frac{1}{\la}\sum_{n \in\Z_{\la}} \bra{n}^{2s}|\wh{w}(t,n)|^2\le C(\norm{u_1}_{X_{\la}^{s,\frac12}}, \norm{u_2}_{X_{\la}^{s,\frac12}}) \norm{w}_{X_{\la}^{s,\frac12}}^2
\end{equation}
fails to hold for any $s \in \R$, when $u_1, u_2 \in X_{\la}^{s,\frac12}$ are solutions to \eqref{eq:lambda-kawahara_u} with $u_1(0) = u_2(0) = u_0$ and $w = u_1 -u_2$. However, the loss of regularity in \eqref{eq:energy_diff} (see Lemma \ref{lem:sym} below) is allowed in the estimate \eqref{eq:nonresonant estimate}, hence we can completely circumvent "the lack of the symmetry" issue. Such an argument was used in the author's previous work \cite{Kwak2018-1}. See Proposition \ref{prop:main_uniqueness} below for more details. 
\end{remark}

\medskip

For $\Sigma_3$, we take the time derivative in $ \wh{\varphi}_{med}(s,n_1)$. 

\medskip

Remark that the estimates below are analogous for the case when we choose another frequency mode in which the time derivative is taken. Thus, we omit the other cases. Furthermore, we extend $u(s)$ from $[0,t]$ to $\R$ with $u(s) = 0$ $(s > t$ or $ s < 0)$ for fixed $0 \le t \le 1$ (we refer to \cite{MN2005} for the precise definition of extension operator).

\medskip

Using \eqref{eq:time derivative}, one has
\[\begin{aligned}
\Sigma_3 =&~{} \frac{C_1}{\la^6}\sup_{N >1}\sup_{n \in I_N} \Bigg|\int_\R \sum_{\N_n^{\la}} \frac{n^2n_1}{H_{\la,0}}(|\wh{u}_{med}(n_1)|^2 - |\wh{u}_{0,med}(n_1)|^2)\wh{u}_{med}(n_1)\wh{u}_{med}(n_2)\wh{u}_{med}(n_3)\wh{u}(-n)  ds\Bigg| \\
&~{}+\frac{C_2}{\la^6}\sup_{N > 1}\sup_{n \in I_N}\Bigg| \int_\R \sum_{\N_n^{\la}}\frac{n^2}{H_{\la,0}}P_{N}\left( n_1\sum_{\N_{n_1}^{\la}} \wh{u}(n_{11})\wh{u}(n_{12})\wh{u}(n_{13}) \right)\wh{u}_{med}(n_2)\wh{u}_{med}(n_3)\wh{u}(-n)  ds\Bigg| \\
=:&~{} \Sigma_{3,1} + \Sigma_{3,2},
\end{aligned}\]
for some constants $C_1, C_2 \in \C$. 

\medskip

For $\Sigma_{3,1}$, we may assume that 
\begin{equation}\label{eq:maximum modulation}
|\tau - p_{\la}(n)| \gtrsim |H_{\la,0}|.
\end{equation}
We replace $\sup_{n \in I_N}$ by $\sum_{n \in I_N}$. Define $g_i$, similarly as before ( $\wh{g}(n) = |\mathcal F_{t}^{-1}(\bra{n}^s|\wt{u}_{med}(n)|)|$ ), and
\begin{equation}\label{h}
\wh{h}(n) = |\mathcal F_{t}^{-1}(\bra{\tau - p_{\la}(n)}^{\frac12}\bra{n}^s|\wt{u}(\tau,n)|)|.
\end{equation}
Since 
\[\frac{1}{\la}\left| |\wh{u}_{med}(n_1)|^2- |\wh{u}_{0,med}(n_1)|^2 \right| \le N^{-2s} (\norm{u_0}_{H_{\la}^s}^2 + \norm{u}_{H_{\la}^s}^2),\]
similarly as in \eqref{eq:nonresonant estimate}, one has from the Sobolev embedding (Lemma \ref{lem:sobolev}) that
\begin{equation}\label{eq:endpoint2}
\begin{aligned}
\Sigma_{3,1} \lesssim&~{} \la^{-\frac32}\int_{\R} (\norm{u_0}_{H_{\la}^s}^2 + \norm{u}_{H_{\la}^s}^2) \sup_{N > 1}N^{-(\frac32 + 6s)}\sum_{|n'| \ge \la^{-1}} \frac{1}{|n'|^{3/2}}\wh{g\overline{g}}(n')\wh{g\overline{h}}(-n') \; ds\\
\lesssim&~{} \la^{-1}\int_{\R}(\norm{u_0}_{H_{\la}^s}^2 + \norm{u}_{H_{\la}^s}^2)\norm{u}_{H_{\la}^s}^3 \norm{\ft^{-1}[\bra{\tau-\mu(n)}^{1/2}\bra{n}^s\wh{u}(n)]}_{L_{\la}^2} \; ds\\
\lesssim&~{} \la^{-1}\left(\norm{u}_{L_t^{10}H_{\la}^s}^5 + \norm{u_0}_{H_{\la}^s}^2\norm{u}_{L_t^{6}H_{\la}^s}^3\right) \norm{\ft^{-1}[\bra{\tau-p(n)}^{1/2}\bra{n}^s\wh{u}(n)]}_{L_{t,\la}^2} \\
\lesssim&~{} \la^{-1}(\norm{u_0}_{H^s}^2 + \norm{u}_{X_{\la}^{s,\frac12}}^2)\norm{u}_{X_{\la}^{s,\frac12}}^4,
\end{aligned}
\end{equation}
whenever $3/2+6s\ge0 \Rightarrow -1/4 \le s < 1/2$.

\medskip

\begin{remark}\label{rem:6-modu}
One has an identity in $\Sigma_{3,1}$ (in particular $|\wh{u}(n_1)|^2\wh{u}(n_1)$ part) that
\[\tau_{1,1} - p_{\la}(n_1) + \tau_{1,2} - p_{\la}(-n_1) + \tau_{1,3} - p_{\la}(n_1)+ \tau_{2} - p_{\la}(n_2) + \tau_{3} - p_{\la}(n_3) = \tau - p_{\la}(n) + H_{\la},\]
which ensures 
\[\max(|\tau_{1,j}-p_{\la}((-1)^{j+1}n_1)|, |\tau_k - p_{\la}(n_k)|, |\tau - p_{\la}(n)| ; j=1,2,3, k=2,3) \gtrsim |H_{\la,0}|.\]
\end{remark}

\begin{remark}\label{rem:choice of maximum modulation}
The estimate \eqref{eq:endpoint2} above does not be affected by the choice of the maximum modulation \eqref{eq:maximum modulation}, thus our choice, in addition to Remarks \ref{rem:6-modu} and \ref{rem:choice of maximum frequency}, does not lose the generality.
\end{remark}

\medskip

For $\Sigma_{3,2}$, we further decompose distributed functions $\wh{u}(n_{1,i})$ into $\wh{u}_{low}(n_{1,i}), \wh{u}_{med}(n_{1,i})$ and $\wh{u}_{high}(n_{1,i})$, $i=1,2,3$. Then, the followings are possible cases (up to the symmetry of frequencies):\footnote{The cases A, B and C are referred to as \emph{high-high-high-high}, \emph{high-high-low-low} and \emph{high-high-high-low}, respectively.}
\[\wh{u}_{med}(n_{11})\wh{u}_{med}(n_{12})\wh{u}_{med}(n_{13}), \tag{Case A}\]
\[\wh{u}_{med}(n_{11})\wh{u}_{low}(n_{12})\wh{u}_{low}(n_{13}), \tag{Case B-1}\]
\[\wh{u}_{med}(n_{11})\wh{u}_{high}(n_{12})\wh{u}_{high}(n_{13}), \tag{Case B-2}\]
\[\wh{u}_{low}(n_{11})\wh{u}_{high}(n_{12})\wh{u}_{high}(n_{13}), \tag{Case B-3}\]
\[\wh{u}_{med}(n_{11})\wh{u}_{med}(n_{12})\wh{u}_{low}(n_{13}), \tag{Case C-1}\]
\[\wh{u}_{high}(n_{11})\wh{u}_{high}(n_{12})\wh{u}_{high}(n_{13}).\tag{Case C-2}\]
Note that the \textbf{Case A} is the worst case, thus we only estimate $\Sigma_{3,2}$ under the restriction of \textbf{Case A}, see Remark \ref{positive}.
\medskip

\textbf{Case A} All comparable frequencies produce the new resonance in the quintic nonlinear interactions. The new resonant function defined by
\begin{equation}\label{eq:resonant relation1} 
\begin{aligned}
\wt{H}_{\la}:=&~{} p_{\la}(n) - p_{\la}(n_2) - p_{\la}(n_3) - \left(p_{\la}(n_{11}) + p_{\la}(n_{12}) + p_{\la}(n_{13})\right)\\
=&~{}\frac52 (n_1+n_2)(n_2+n_3)(n_3+n_1)\left(n_1^2 + n_2^2 + n_3 ^2 + n^2 + \frac65\beta\la^{-2} \right) \\
&~{}+ \frac52 (n_{11}+n_{12})(n_{12}+n_{13})(n_{13}+n_{11})\left(n_{11}^2 + n_{12}^2 + n_{13} ^2 + n_1^2 + \frac65\beta\la^{-2} \right)\\
&~{}-\frac{\mu n}{(2\pi \lambda)^2} |\wh{v}_{\lambda, 0}(n)|^2 + \sum_{j=1}^3 \frac{\mu n_j}{(2\pi \lambda)^2} |\wh{v}_{\lambda, 0}(n_j)|^2\\
&~{}-\frac{\mu n_1}{(2\pi \lambda)^2} |\wh{v}_{\lambda, 0}(n_1)|^2 + \sum_{j=1}^3 \frac{\mu n_{1,j}}{(2\pi \lambda)^2} |\wh{v}_{\lambda, 0}(n_{1,j})|^2
\end{aligned}
\end{equation}
does not guarantee the smoothing effect when the frequencies satisfy, for instance,
\begin{equation}\label{sextic resonance}
\begin{array}{lll}
n_1= N+a,& \quad n_2 = -N-a-b,& \quad n_3 = N+b,\\
n_{11} = N +a+b,& \quad n_{12}=-N-b,& \quad n_{13} = N,
\end{array}
\end{equation}
where $a,b \in \Z_{\la}$ with $a \neq 0$, $b \neq 0$ and $a+b \neq 0$. Hence, we do not have any advantage from the \emph{dispersive smoothing effect}
\[L_{max} :=\max(|\tau_{1k} - p_{\la}(n_{1k})|,|\tau_j - p_{\la}(n_j)|, |\tau - p_{\la}(n)| ; k=1,2,3, j= 2,3) \gtrsim |\wt{H}_{\la}|.\]

\medskip

We assume $|n_1+n_2| \sim N$. Let us define 
\begin{equation}\label{f}
\wt{f}(m) = \bra{m}^s|\wt{u}_{med}(m)|.
\end{equation} 
Then it suffices to consider
\begin{equation}\label{Sigma33}
\la^{-6}N^{-6s}\sup_{n \in I_N} \sum_{\N_n^{\la}} \sum_{\N_{n_1}^{\la}} \frac{1}{|n_2+n_3||n_3+n_1|}\wh{f}(n_{11})\wh{f}(n_{12})\wh{f}(n_{13})\wh{f}(n_2)\wh{f}(n_3)\wh{f}(-n),
\end{equation}
since $|H_{\la,0}| \gtrsim |n_1+n_2||n_2+n_3||n_3+n_1|N^2$. We change the variables as follows:
\begin{equation}\label{CoV}
\begin{array}{l}
n_2 + n_3 = n'\\
n_{12} + n_{13} = n''
\end{array}
\Longrightarrow
\begin{array}{lll}
n_2 = n_2,& \quad n_3 = n'-n_2,& \\
n_{11} = n-n'-n'',& \quad n_{12} = n_{12},& \quad n_{13} = n'' - n_{12}.
\end{array}
\end{equation}
A direct computation yield
\begin{equation}\label{eq:estimation-1}
\begin{aligned}
\eqref{Sigma33} =& \la^{-6}N^{-6s}\sup_{n \in I_N}\sum_{\substack{n',n'',n_2,n_{12} \\ |n'|, |n-n_2| \ge \la^{-1} }} \frac{1}{|n'||n-n_2|} \\
& \hspace{5em} \times \wh{f}(n-n'-n'')\wh{f}(n_{12})\wh{f}(n''-n_{12})\wh{f}(n_2)\wh{f}(n'-n_2)\wh{f}(n)\\
=& \la^{-5}N^{-6s}\sup_{n \in I_N}\sum_{\substack{n',n'',n_2 \\ |n'|, |n-n_2| \ge \la^{-1} }} \frac{1}{|n'||n-n_2|} \\
& \hspace{5em} \times \wh{f}(n-n'-n'')\wh{f^2}(n'')\wh{f}(n_2)\wh{f}(n'-n_2)\wh{f}(-n)\\
\lesssim& \la^{-2}N^{-6s} \norm{f}_{L_{\la}^1}\norm{f^2}_{L_{\la}^2}\norm{f}_{L_{\la}^2}\left(\frac{1}{\la}\sum_{|n'|\ge\la^{-1}} |n'|^{-2}\right)^{\frac12}\\
&\hspace{5em}\times \left(\frac{1}{\la^2}\sum_{n_2,n'}|\wh{f}(n_2)\wh{f}(n'-n_2)|^2\right)^{\frac12}\sup_{n \in I_N}\left(\frac{1}{\la}\sum_{|n_2-n|\ge\la^{-1}} |n_2-n|^{-2}\right)^{\frac12}\\
\lesssim& \la^{-\frac32}N^{-6s} \norm{f}_{L_{\la}^4}^2\norm{f}_{L_{\la}^2}^4.
\end{aligned}
\end{equation}
Thus, $L^4$ Strichartz estimate (Lemma \ref{lem:L4}) and Sobolev embedding (Lemma \ref{lem:sobolev}) yield
\begin{equation}\label{eq:estimation-3}
\Sigma_{3,2} \lesssim \la^{-\frac32}\sup_{N \ge 1} N^{-6s}\norm{f}_{L_{t,\la}^4}^2\norm{f}_{L_t^8(L_{\la}^2)}^4 \lesssim \la^{-\frac32}\norm{u}_{X_{\la}^{s,\frac12}}^6,
\end{equation}
whenever $0 \le s < \frac12$.

\begin{remark}\label{WLOG}
An analogous computation ensures that the assumption $|n_1+n_2| \sim N$ does not lose the generality. Indeed, using the following change of variables
\[\begin{array}{l}
n_1 + n_2 = n'\\
n_{12} + n_{13} = n''
\end{array}
\Longrightarrow
\begin{array}{lll}
n_1 = n'-n_2,& \quad n_2 = n_2,& \quad n_3 = n-n', \\
n_{11} = n'-n_2-n'',& \quad n_{12} = n_{12},& \quad n_{13} = n'' - n_{12}.
\end{array}\]
for the case when $|n_2+n_3| \sim N$, or performing \eqref{CoV} for the case when $|n_3 + n_1| \sim N$, we can follow \eqref{eq:estimation-1}, and hence we obtain \eqref{eq:estimation-3}.
\end{remark}

\begin{remark}\label{rem:threshold}
The $\mathcal{F}\ell^{\infty}$-smoothing estimate enables to achieve the local well-posedness at the endpoint regularity $s=0$, while the $\mathcal{F}\ell^1$-smoothing estimate holds on only sub-critical regularity regime $s > 0$. See \cite{Kwak2018-1, MPV2018} and \cite{TT2004, NTT2010, MT2018, OW2018} for the comparison. This approach is more important in this paper than others, since the local well-posedness in $L^2$ immediately ensures the global one, thanks to the conservation law \eqref{eq:L2}.

\medskip

On the other hand, our proof fails to control $\Sigma_{3,2}$ below $L^2$, due to \eqref{sextic resonance}. See also Remark 3.2 in \cite{NTT2010} for the similar phenomenon in the context of mKdV.
\end{remark}

\begin{remark}\label{rem:loss of regularity}
In \eqref{Sigma33}, once replacing $\sup_{n \in I_N}$ by $\sum_{n \in I_{N}}$, we will complete \eqref{eq:estimation-1} but with $0+$-derivative loss. More precisely, for $\underline{f} = \mathcal F_x^{-1}(|\wh{f}(n)|)$ ($\norm{f}_{L_{\la}^2} = \norm{\underline{f}}_{L_{\la}^2}$), the change the variables ($n_1+n_3 = n'$ and $n_{12} + n_{13} = n''$) yields
\[\begin{aligned}
&\frac{1}{\la^6}N^{-6s}\sum_{n \in I_N} \sum_{\N_n^{\la}} \sum_{\N_{n_1}^{\la}} \frac{1}{|n_2+n_3||n_3+n_1|}\wh{f}(n_{11})\wh{f}(n_{12})\wh{f}(n_{13})\wh{f}(n_2)\wh{f}(n_3)\wh{f}(-n)\\
\lesssim& \la^{-2}N^{-6s}(\log \la + \log N)\norm{f}_{L_{\la}^2}^2\norm{f^2}_{L_{\la}^2}\norm{\underline{f}^2}_{L_{\la}^1},
\end{aligned}\]
which implies
\[\frac{1}{\la^2}\sum_{n \in I_N}|n|\left||\wh{u}(n)|^2 - |\wh{u}_0(n)|^2\right| \lesssim \la^{-2+}N^{-6s}N^{0+}\norm{u}_{X_{\la}^{2,\frac12}}^6.\]
Analogously we have the $0+$ regularity loss below, if $\sup_{n \in I_N}$ is changed by $\sum_{n \in I_N}$. Consequently, even localized $\mathcal{F}\ell^1$ estimate seems not be enough for \emph{a priori} bound in $L^2$ level, but is available to show \emph{tightness}, i.e.,
\[\lim_{N \to \infty}\sum_{|n| \ge N}|\wh{u}(n)|^2 = 0.\]
See \eqref{eq:a priori 4}. 
\end{remark}

\begin{remark}\label{positive}
One can see that $\Sigma_{3,2} \lesssim \la^{-\frac12}\norm{u}_{X_{\la}^{s,\frac12}}^6$ holds for $s \ge 0$ not only under \textbf{Case A}, but also \textbf{Cases B--C}, since $\bra{n}^{-s} \lesssim N^{-s}$ holds for all cases, and the computation \eqref{eq:estimation-1} does not depend on the frequency supports.  
\end{remark}

\begin{remark}
One can control $\Sigma_{3,2}$ under the restrictions of \textbf{Case B--C}, for $s > -\frac13$. However, the smoothing effect for the negative regularity is no longer available in the proof of Local well-posedness, since $L^4$ estimate is not valid for $H_{\la}^s$ data, $s < 0$, thus we do not pursue it here.
\end{remark}
\bigskip

\textbf{Case II.} We deal with the \emph{high-low} interactions. Thanks to the symmetry, we may assume that $|n_1| \ge |n_2| \ge |n_3|$. Moreover, it is enough to consider the \emph{high $\times$ low $\times$ low $\Rightarrow$ high} interaction case, due to $n^2$ in \eqref{eq:hl0} below (see also Remark \ref{rem:low resulting frequency}). 

\medskip

We first address the regularity $s > 0$. We know from \eqref{eq:resonant function la} that
\[|H_{\la}| \gtrsim |n_2+n_3|N^4.\]
Without loss of generality\footnote{The estimate \eqref{eq:low result} below does not depend on the choice of the maximum modulation.}, we assume $|\tau - p_{\la}(n)| \gtrsim |H_{\la}|$. For
\[L_{max}^{-1} \lesssim |n_2+n_3|^{-\frac12}N^{-2},\]
we use the notations $f$ and $h$ defined as in \eqref{f} and \eqref{h}, respectively. We may assume that $|n_j| > 1$, otherwise we can cover the endpoint regularity $s =0$ as well. Thus, it suffices to estimate
\begin{equation}\label{eq:hl0}
\frac{1}{\la^4}\sup_{N>1}N^{-4s}\sup_{n \in I_N}\int_{\R}  \sum_{\N_n^{\la}} \frac{1}{|n_2+n_3|^{\frac12}}\wh{f}(n_1)\wh{f}(n_2)\wh{f}(n_3)\wh{h}(-n).
\end{equation}
For $s > 0$, a straightforward computation, in addition to the change of variables ($n_2+n_3 = n'$) and the Sobolev embedding (Lemma \ref{lem:sobolev}), gives
\begin{equation}\label{eq:low result}
\begin{aligned}
\eqref{eq:hl0}=&~{} \frac{1}{\la^4}\sup_{N>1}N^{-4s}\sup_{n \in I_N}\int_{\R}\sum_{\substack{n_2,n' \\ \la^{-1} \le |n'| \le N}} \frac{1}{|n'|^{\frac12}}\wh{f}(n-n')\wh{f}(n_2)\wh{f}(n'-n_2)\wh{h}(-n)\\
=&~{} \frac{1}{\la^3}\sup_{N>1}N^{-4s}\sup_{n \in I_N}\int_{\R}\sum_{\la^{-1} \le |n'| \le N} \frac{1}{|n'|^{\frac12}}\wh{f}(n-n')\wh{f^2}(n')\wh{h}(-n)\\
\lesssim&~{} \frac{1}{\la^2}\int_{\R} \sup_{N > 1} N^{-4s}(\log \la + \log N)\norm{f}_{L_{\la}^2}^3\norm{h}_{L_{\la}^2}\\
\lesssim&~{} \frac{1}{\la^2}\log \la \norm{f}_{L_t^6(L_{\la}^2)}^3\norm{h}_{L_{t,\la}^2}\\
\lesssim&~{} \frac{1}{\la^2}\log \la \norm{u}_{X_{\la}^{s,\frac12}}^4,
\end{aligned}
\end{equation}
whenever $s > 0$. 
\begin{remark}\label{rem:low resulting frequency}
The \emph{high $\times$ high $\times$ low $\Rightarrow$ low} interaction case can be dealt with by \eqref{eq:low result}, thanks to
\[\frac{n^2}{\bra{n_1}^s\bra{n_2}^s\bra{n_3}^s\bra{n}^sL_{\max}^{\frac12}} \lesssim \frac{1}{N^{2s}|n_1+n_2|^{\frac12}},\]
for $s > 0$.
\end{remark}

Now we address the end point regularity $s = 0$. In this case, we cannot obtain
\begin{equation}\label{eq:hhll0}
\frac{1}{\la^4}\sup_{N>1}\sup_{n \in I_N}\left| \int_{\R} n^2\sum_{\N_n^{\la}}  \wh{u}_{med}(n_1)\wh{u}_{low}(n_2)\wh{u}_{low}(n_3)\wh{u}(-n)  \; ds\right| \lesssim \norm{u}_{X_{\la}^{0,\frac12}},
\end{equation}
due to the logarithmic divergence appearing in \eqref{eq:low result}. To overcome it we again use the normal form reduction method. For the sake of simplicity, we do not push the regularity $s$ below $0$. Similarly as \eqref{eq:normal form}, we have
\begin{equation}\label{eq:normal form_hhll}
\begin{aligned}
\mbox{LHS of }\eqref{eq:hhll0} \le&~{} \frac{1}{\la^4}\sup_{N>1}\sup_{n \in I_N}\sum_{\N_n} \frac{n^2}{|H_{\la}|}\left| \wh{u}_{med}(t,n_1)\wh{u}_{low}(t,n_2)\wh{u}_{low}(t,n_3)\wh{u}(t,-n) \right|\\
&~{} + \frac{1}{\la^4}\sup_{N>1}\sup_{n \in I_N} \sum_{\N_n}\frac{n^2}{|H_{\la}|} \left|\wh{u}_{0,med}(n_1)\wh{u}_{0,low}(n_2)\wh{u}_{0,low}(n_3)\wh{u}_{0}(-n)\right|\\
&~{}+\frac{1}{\la^4}\sup_{N>1}\sup_{n \in I_N} \Bigg|\int_0^t \sum_{\N_n} \frac{n^2e^{-is(H_{\la})}}{iH_{\la}}\\
& \hspace{7em} \times \frac{d}{ds}\left(\wh{\varphi}_{med}(s,n_1)\wh{\varphi}_{low}(s,n_2)\wh{\varphi}_{low}(s,n_3)\wh{\varphi}(s,-n) \right) ds \Bigg|\\
=:&~{} \Xi_1 + \Xi_2 + \Xi_3.
\end{aligned}
\end{equation}
Remark that a direct computation gives
\[|H_{\la}| \gtrsim |n_2+n_3|N^4,\]
which is stronger than one in \textbf{Case I}. Thus, $\Xi_1$ and $\Xi_2$ are controlled, similarly as the estimates of $\Sigma_1$ and $\Sigma_2$, by
\[\Xi_1 + \Xi_2  \lesssim \la^{-2}\log \la\norm{u_0}_{L_{\la}^2}^4 + \la^{-2}\log \la\left(\norm{u_0}_{L_{\la}^2}^2 + \la^{\frac12} \norm{u}_{X_{\la}^{0,\frac12}}^4\right)^2.\]
Remark \ref{rem:boundary} is available to this estimate for the difference of two solutions. 

\medskip

Taking the time derivative to $n_1$ mode, one has 
\[\begin{aligned}
\Xi_3 =&~{} \frac{C_1}{\la^6}\sup_{N >1}\sup_{n \in I_N} \Bigg|\int_0^t \sum_{\N_n^{\la}} \frac{n^2n_1}{H_{\la,0}}(|\wh{u}_{med}(n_1)|^2 - |\wh{u}_{0,med}(n_1)|^2)\wh{u}_{med}(n_1)\wh{u}_{low}(n_2)\wh{u}_{low}(n_3)\wh{u}(-n)  ds\Bigg| \\
&~{}+\frac{C_2}{\la^6}\sup_{N > 1}\sup_{n \in I_N}\Bigg| \int_0^t \sum_{\N_n^{\la}}\frac{n^2}{H_{\la,0}}P_{N}\left( n_1\sum_{\N_{n_1}^{\la}} \wh{u}(n_{11})\wh{u}(n_{12})\wh{u}(n_{13}) \right)\wh{u}_{low}(n_2)\wh{u}_{low}(n_3)\wh{u}(-n)  ds\Bigg| \\
=:&~{} \Xi_{3,1} + \Xi_{3,2},
\end{aligned}\]
The replacement $\sup_{n \in I_N}$ by $\sum_{n \in I_N}$ is available for $\Xi_{3,1}$. The same computation as in \eqref{eq:low result} (but, here $\wh{f}(n) = |\wh{u}(n)|$) yields
\[\begin{aligned}
\Xi_{3,1} \lesssim&~{} \la^{-6}\int_0^t \left(\norm{u}_{L^2}^2+\norm{u_0}_{L^2}^2\right)\sup_{N>1}N^{-1}\sum_{\substack{n_1, n_2, n' \\ \la^{-1} \le |n'| \le N}} \frac{1}{|n'|}\wh{f}(n_1)\wh{f}(n_2)\wh{f}(n'-n_2)\wh{f}(-n_1 - n') \; ds\\
\lesssim&~{} \la^{-4}\log \la\int_0^t \left(\norm{u}_{L_{\la}^2}^2+\norm{u_0}_{L_{\la}^2}^2\right) \norm{f}_{L_{\la}^2}^4,
\end{aligned}\]
which, in addition to Lemma \ref{lem:sobolev}, implies
\[\Xi_{3,1} \lesssim \la^4 \log \la \left(\norm{u}_{X_{\la}^{0,\frac12}}^2+\norm{u_0}_{L^2}^2\right)\norm{u}_{X_{\la}^{0,\frac12}}^4.\]

\medskip

On the other hand, one can split the frequency relation among $n_{11}$, $n_{12}$ and $n_{13}$ into \textbf{Case A}--\textbf{Case C}. Under the relation presented in \textbf{Case B-1}, one cannot have an additional smoothing effect ($L_{max} \gtrsim |\wt{H}|$) in $\Xi_{3,2}$, where $\wt{H}_{\la}$ is defined as in \eqref{eq:resonant relation1}, due to the same reason in $\Sigma_{3,2}$ under \textbf{Case A}. However, \eqref{eq:estimation-1} is still available for $\Xi_{3,2}$ under \textbf{Case B-1}, thus we handle this case. In the other cases, the stronger $|H_{\la}|$ and additional dispersive smoothing effects enables us to estimate $\Xi_{3,2}$ more easily or similarly as the estimate of $\Sigma_{3,2}$. Thus we have 
\[\Xi_{3,2} \lesssim  \la^{-\frac32}\norm{u}_{X_{\la}^{0,\frac12}}^6.\]

\medskip

Contributions from the time derivative taken in the other modes in $\Xi_3$ could be dealt with similarly or easily, due to $|n_2|, |n_3| \ll |n_1| \sim |n|$. We skip the details. We remark that all computations established as in \textbf{Case I} are available for \textbf{Case II}, when $s=0$.

\medskip

The argument used in \textbf{Case II} always holds\footnote{A direct calculation for $s>0$ and the normal form method for $s =0$ are needed.} under the \emph{high $\times$ high $\times$ low $\Rightarrow$ high}\footnote{Remark \ref{rem:low resulting frequency} allows to deal with the \emph{high $\times$ high $\times$ high $\Rightarrow$ low} case by the same argument.} interaction case, we hence obtain the same result as in \textbf{Case II}. 

\medskip

Gathering all results in \textbf{Cases I, II} and \textbf{III}, we complete the proof of \eqref{eq:main}.  

\end{proof}

As an immediate corollary, we have 
\begin{corollary}\label{cor:main}
Let $\lambda \ge 1$. Assume that $\norm{v_{\lambda,0}}_{L_{\la}^2} \le \rho$, for a sufficient small $0 < \rho \ll 1$. Let $0 \le s < \frac12$, $t \in [0,1]$ and $u_0 \in C^{\infty}(\T_{\la})$. Suppose that $u$ is a real-valued smooth solution to \eqref{eq:lambda-kawahara_u} and $u \in X_{\la}^{s,\frac12}$. Then the following estimate holds:
\[\begin{aligned}
&\frac{1}{\la^2}\sup_{n \in \Z_{\la}} |n|\left||\wh{u}(t,n)|^2 - |\wh{u}_0(n)|^2 \right|\\
\lesssim&~{} \la^{-1}\log \la\left(\norm{u_0}_{H_{\la}^s}^4 + \left(\norm{u_0}_{H_{\la}^{s}}^2 + \la^{\frac12+2s} \norm{u}_{X_{\la}^{s,\frac12}}^4\right)^2\right) + \la^{-\frac32}\norm{u}_{X_{\la}^{s,\frac12}}^4 + \la^{-\frac32}\norm{u}_{X_{\la}^{s,\frac12}}^6.
\end{aligned}\]
Note that $\log \la$ is replaced by $1$ when $\la =1$.
\end{corollary}

\begin{proof}
The proof follows from
\[\begin{aligned}
&\frac{1}{\la^2}\sup_{n \in \Z_{\la}} |n|\left||\wh{u}(t,n)|^2 - |\wh{u}_0(n)|^2 \right|\\
=&~{}C\sup_{n \in \Z_{\lambda}}\Bigg|\frac{1}{\lambda^4}\mbox{Im}\Bigg[ \int_0^t  n^2 \sum_{\N_n^{\la}}\wh{u}(s,n_1)\wh{u}(s,n_2)\wh{u}(s,n_3)\wh{u}(s,-n) \; ds \Bigg]\Bigg|
\end{aligned}\]
and Proposition \ref{prop:main}.
\end{proof}

\subsection{Difference of two solutions}
Let $\lambda \ge 1$. Assume that $\norm{v_{\lambda,0}}_{L_{\la}^2} \le \rho$, for a sufficient small $0 < \rho \ll 1$. Let $u_1, u_2$ be solutions to \eqref{eq:lambda-kawahara_u} with the same initial data $u_1(0) = u_2(0) = u_0$. Let $w = u_1 - u_2$. Then $w$ satisfies
\begin{equation}\label{eq:kawahara3}
\begin{aligned}
\pt \wh{w}(n) - ip_{\la}(n) \wh{w}(n) =&~{} -\frac{\mu i n}{(2\pi \lambda)^2}(|\wh{u}_1(n)|^2-|\wh{u}_0(n)^2|)\wh{w}(n)\\
&~{}-\frac{\mu i n}{(2\pi \lambda)^2}(|\wh{u}_1(n)|^2 - |\wh{u}_2(n)|^2)\wh{u}_2(n)\\
&~{}+  \frac{\mu i n}{3(2\pi \lambda)^2} \sum_{\N_n^{\la}}\wh{F}(u_1,u_2),
\end{aligned}
\end{equation}
where
\begin{equation}\label{F}
\wh{F}(u_1,u_2) = \wh{w}(n_1)\wh{u}_1(n_2)\wh{u}_1(n_3)+\wh{u}_2(n_1)\wh{w}(n_2)\wh{u}_1(n_3)+\wh{u}_2(n_1)\wh{u}_2(n_2)\wh{w}(n_3).
\end{equation}

\medskip

Corollary \ref{cor:main} and Lemma \ref{lem:nonres-trilinear} enable us to handle the first and third terms in the right-hand side of \eqref{eq:kawahara3}. Thus, it remains to control $\frac{n}{\la^2} \left( |\wh{u}_1(n)|^2 - |\wh{u}_2(n)|^2 \right)$ in the resonant terms. Using \eqref{eq:energy}, one reduces to dealing with
\begin{equation}\label{eq:res_tri}
\begin{aligned}
\frac{1}{\la^4}\int_0^t n^2\sum_{\N_n^{\la}}\bigg[&\wh{w}(n_1)\wh{u}_1(n_2)\wh{u}_1(n_3)\wh{u}_1(-n)+\wh{u}_2(n_1)\wh{w}(n_2)\wh{u}_1(n_3)\wh{u}_1(-n) \\
&+\wh{u}_2(n_1)\wh{u}_2(n_2)\wh{w}(n_3)\wh{u}_1(-n) + \wh{u}_2(n_1)\wh{u}_2(n_2)\wh{u}_2(n_3)\wh{w}(-n)\bigg] \; ds.
\end{aligned}
\end{equation}

We, without loss of generality, choose the second term in \eqref{eq:res_tri} in order to state and prove the main proposition in this section.\begin{proposition}\label{prop:main_uniqueness}
Let $\lambda \ge 1$. Assume that $\norm{v_{\lambda,0}}_{L_{\la}^2} \le \rho$, for a sufficient small $0 < \rho \ll 1$. Let $0 \le s < \frac12$, $t \in [0,1]$ and $u_0 \in C^{\infty}(\T_{\la})$. Suppose that $u_1$ and $u_2$ are a real-valued smooth solution to \eqref{eq:lambda-kawahara_u} with $u_1(0) = u_2(0) = u_0$ and $u_1, u_2 \in X_{\la}^{s,\frac12}$. Let $w = u_1 - u_2$. Then the following estimate holds:
\[\begin{aligned}
\frac{1}{\la^4}\sup_{n \in \Z} \Bigg| \int_0^t  n^2 \sum_{\N_n^{\la}}\wh{u}_2(s,n_1)\wh{w}(s,n_2)&\wh{u}_1(s,n_3)\wh{u}_1(s,-n) \; ds \Bigg|\\
\lesssim&~{} C(\norm{u_0}_{H_{\la}^s}, \norm{u_1}_{X_{\la}^{s,\frac12}}, \norm{u_2}_{X_{\la}^{s,\frac12}}) \norm{w}_{X_{\la}^{s,\frac12}}.
\end{aligned}\]
\end{proposition}
\begin{remark}
The proof of Proposition \ref{prop:main_uniqueness} basically follows from the proof of Proposition \ref{prop:main}. The only difference is to estimate the boundary terms generated in the normal form process, as mentioned in Remark \ref{rem:boundary}. We only point this difference out in the proof of Proposition \ref{prop:main_uniqueness} below. 
\end{remark}

In order to handle the difficulty arising from the lack of the symmetry, we need the following lemma:
\begin{lemma}\label{lem:sym}
Let $\la \ge 1$. Assume that $\norm{v_{\lambda,0}}_{L_{\la}^2} \le \rho$, for a sufficient small $0 < \rho \ll 1$. Let $0 \le s < \frac12$, $t \in [0,1]$ and $u_0 \in C^{\infty}(\T_{\la})$. Suppose that $u_1$ and $u_2$ are solutions to \eqref{eq:lambda-kawahara_u} on $[-1,1]$ with $u_{1,0} = u_{2,0}$, and $u, v \in X_T^{s,\frac12}$. Let $w = u_1-u_2$. Then the following estimate holds:
\begin{equation}\label{eq:sym-1}
\norm{w(t)}_{H_{\la}^{-\frac12}}^2 \lesssim \la^{\frac12}(\norm{u_1}_{X_{\la}^{s,\frac12}}^2 + 2\norm{u_1}_{X_{\la}^{s,\frac12}}\norm{u_2}_{X_{\la}^{s,\frac12}} + \norm{u_2}_{X_{\la}^{s,\frac12}}^2)\norm{w}_{X_{\la}^{s,\frac12}}^2.
\end{equation}
\end{lemma}

\begin{proof}
Using \eqref{eq:kawahara3}, a direct calculation gives
\[\begin{aligned}
\pt|\wh{w}(n)|^2 =&~{} -\frac{2\mu i n}{(2\pi \lambda)^2}\mbox{Im}\left[\wh{u}_1(n)\wh{w}(-n)\wh{u}_2(n)\wh{w}(-n) \right] \\
&+\frac{2\mu i n}{3(2\pi \lambda)^2}\mbox{Im}\left[\sum_{\N_n^{\la}}\wh{F}(u_1,u_2)\wh{w}(-n) \right]\\
&=:A(t,n)+B(t,n),
\end{aligned}\]
for $\wh{F}(u_1,u_2)$ as in \eqref{F}. One immediately obtains
\[\sum_{n \in \Z_{\la}} \wh{u}_1(n)\wh{w}(-n)\wh{u}_2(n)\wh{w}(-n) \lesssim \la^2\norm{u_1(t)}_{L_{\la}^2}\norm{u_2(t)}_{L_{\la}^2}\norm{w(t)}_{L_{\la}^2}^2.\]
Hence, the H\"older inequality and Lemma \ref{lem:sobolev} yield
\[\begin{aligned}
\int_0^t \frac{1}{\la} \sum_{n\in \Z_{\la}} \bra{n}^{-1} A(s,n) \; ds &\lesssim \la^{-1} \norm{u_1}_{L_t^4L_{\la}^2}\norm{u_2}_{L_t^4L_{\la}^2}\norm{w}_{L_t^4L_{\la}^2}^2 \\
&\lesssim \la^{-1} \norm{u_1}_{X_{\la}^{0,\frac14}}\norm{u_2}_{X_{\la}^{0,\frac14}}\norm{w}_{X_{\la}^{0,\frac14}}^2.
\end{aligned}\]
On the other hand, by Lemma \ref{lem:nonres-trilinear}, we have
\[\begin{aligned}
\int_0^t \frac{1}{\la}\sum_{n\in \Z_{\la}} \bra{n}^{-1} B(s,n) \; ds \lesssim&~{} \int_0^t \frac{1}{\la}\sum_{n\in \Z_{\la}} B(s,n) \; ds \\
\lesssim&~{} \la^{\frac12}(\norm{u_1}_{X_{\la}^{0,\frac12}}^2 + \norm{u_1}_{X_{\la}^{0,\frac12}}\norm{u_2}_{X_{\la}^{0,\frac12}} + \norm{u_2}_{X_{\la}^{0,\frac12}}^2)\norm{w}_{X_{\la}^{0,\frac12}}^2.
\end{aligned}\]
Collecting all, one proves \eqref{eq:sym-1}.
\end{proof}

\begin{proof}[Proof of Proposition \ref{prop:main_uniqueness}]
In view of the proof of Proposition \ref{prop:main}, as mentioned again, our analysis does not rely on the symmetry of functions (or the structure of equation \eqref{eq:lambda-kawahara_u}), except for the estimate of the boundary term in the normal form process, in particular, an application of Corollary \ref{cor:trilinear} in \eqref{eq:boundary estimate}. Thus, we are going to show how to deal with this case compared to the estimates \eqref{eq:nonresonant estimate} and \eqref{eq:boundary estimate}.

\medskip

The normal form argument in addition to Remark \ref{rem:negligible factor} reduces to dealing with (see \eqref{eq:normal form})\footnote{The boundary term at $s=0$ ($\Sigma_2$ in \eqref{eq:normal form}) does not appear, due to $w(0,x) = 0$.}
\[\begin{aligned}
 & \frac{1}{\la^4}\sup_{N>1}\sup_{n \in I_N}  \sum_{\N_n^{\la}} \frac{n^2}{|H_{\la}|}\left| \wh{u}_2(t,n_1)\wh{w}(t,n_2)\wh{u}_1(t,n_3)\wh{u}_1(t,-n) \right|\\
+&~{} \frac{1}{\la^4}\sup_{N>1}\sup_{n \in I_N} \Bigg|\int_0^t \sum_{\N_n^{\la}} \frac{n^2}{iH_{\la}} \frac{d}{ds}\Big[\big(e^{-isp_{\la}(n_1)} \wh{u}_2(s,n_1)\big) \big(e^{-isp_{\la}(n_2)} \wh{w}(s,n_2)\big)\\
&~{} \hspace{14em} \times  \big(e^{-isp_{\la}(n_3)} \wh{u}_1(s,n_3) \big) \big(e^{isp_{\la}(n)} \wh{u}_1(s,-n) \big) \Big]\; ds \Bigg|\\
=:&~{} \wt{\Sigma}_1 + \wt{\Sigma}_3.
\end{aligned}\]
where $H_{\la}$ is defined as in \eqref{eq:resonant function la}, and $\wh{u}_1, \wh{u}_2$ and $\wh{w}$ are supported in $I_N$. The estimate of $\wt{\Sigma}_3$ is analogously dealt with as the estimate of $\Sigma_3$ in \textbf{Case I} in the proof of Proposition \ref{prop:main}. Indeed, using \eqref{eq:time derivative} for $u_1$ and $u_2$, or 
\[\begin{aligned}
&\pt \left( e^{-itp_{\la}(n)}\wh{w}(n) \right)\\
=&~{} e^{-itp_{\la}(n)} \left( \pt \wh{w}(n) - ip_{\la}(n) \wh{w}(n) \right)\\
=&~{}  -\frac{\mu i n}{(2\pi \lambda)^2}e^{-itp_{\la}(n)} \Bigg( (|\wh{u}_1(n)|^2-|\wh{u}_0(n)^2|)\wh{w}(n) + (|\wh{u}_1(n)|^2 - |\wh{u}_2(n)|^2)\wh{u}_2(n) \\
&+\frac{\mu i n}{3(2\pi \lambda)^2}  \sum_{\N_n^{\la}}\left[\wh{w}(n_1)\wh{u}_1(n_2)\wh{u}_1(n_3)+\wh{u}_2(n_1)\wh{w}(n_2)\wh{u}_1(n_3)+\wh{u}_2(n_1)\wh{u}_2(n_2)\wh{w}(n_3)\right] \Bigg),
\end{aligned}\]
for $w$, one can apply the exact same arguments used in \textbf{Case A--C} to $\wt{\Sigma}_3$ to obtain
\[\wt{\Sigma}_3 \lesssim \la^{-\frac12}C(\norm{u_1}_{X_{\la}^{s,\frac12}}, \norm{u_2}_{X_{\la}^{s,\frac12}}) \norm{w}_{X_{\la}^{s,\frac12}}.\]
Thus, it suffices to estimate $\wt{\Sigma}_1$. Compared to \eqref{eq:nonresonant estimate}, we perform an unfair distribution of derivatives to use Lemma \ref{lem:sym}. Let 
\[\begin{aligned}
&\wh{g}_1(n_1) = |\wh{u}_2(n_1)|, \hspace{2em} \wh{g}_2(n_2) = \bra{n_2}^{-\frac12}|\wh{w}(n_2)|;\\ 
&\wh{g}_3(n_3) = |\wh{u}_1(n_3)|, \hspace{2em} \wh{g}_4(-n) = |\wh{u}_1(-n)|.
\end{aligned}\]
We assume $|H|^{-1} \lesssim \la N^{-3}|n_1+n_2|^{-1}$\footnote{This assumption does not lose the generality, see Remark \ref{rem:choice of maximum frequency}.}. The change of variable $n'=n_1+n_2 \neq 0$ and a direct computation\footnote{In this case, we replace $\sup_{n \in I_N}$ by $\sum_{n \in I_N}$} yield
\begin{equation}\label{Sigma3}
\begin{aligned}
\wt{\Sigma}_1 \lesssim&~{} \la^{-3}\sup_{N\ge1}N^{-\frac12}\sum_{\substack{n_1,n_4\\ |n'| \ge \la^{-1}}} \frac{1}{|n'|}\wh{g}_1(n_1)\wh{g}_2(n'-n_1)\wh{g}_3(n_4-n')\wh{g}_4(-n_4)\\
\lesssim&~{} \la^{-1}\log \la \norm{u_2}_{L_{\la}^2}\norm{u_1}_{L_{\la}^2}^2 \norm{w}_{H_{\la}^{-\frac12}}.
\end{aligned}
\end{equation}
Corollary \ref{cor:trilinear} and Lemma \ref{lem:sym} enable to estimate the last terms in \eqref{Sigma3}, and hence we have
\[\wt{\Sigma}_1 \lesssim C(\la^{0+}, \norm{u_0}_{H_{\la}^s}, \norm{u_1}_{X_{\la}^{s,\frac12}}, \norm{u_2}_{X_{\la}^{s,\frac12}}) \norm{w}_{X_{\la}^{s,\frac12}},\]
for $s \ge 0$. An analogous argument holds true for $\wt{\Xi}_1$, which can be similarly defined as in \eqref{eq:normal form_hhll}. Thus, it completes the proof of Proposition \ref{prop:main_uniqueness}.
\end{proof}
From \eqref{eq:res_tri}, one immediately has
\begin{corollary}\label{cor:main_uniqueness}
Let $\lambda \ge 1$. Assume that $\norm{v_{\lambda,0}}_{L_{\la}^2} \le \rho$, for a sufficient small $0 < \rho \ll 1$. Let $0 \le s < \frac12$, $t \in [0,1]$ and $u_0 \in C^{\infty}(\T_{\la})$. Suppose that $u_1$ and $u_2$ are a real-valued smooth solution to \eqref{eq:lambda-kawahara_u} with $u_1(0) = u_2(0) = u_0$ and $u_1, u_2 \in X_{\la}^{s,\frac12}$. Let $w = u_1 - u_2$. Then the following estimate holds:
\[\begin{aligned}
\frac{1}{\la^2}\sup_{n \in \Z_{\la}} |n|\Big||\wh{u}_1(n)|^2-|\wh{u}_2(n)|^2\Big| \le C(\la^{0+}, \norm{u_0}_{H_{\la}^s}, \norm{u_1}_{X_{\la}^{s,\frac12}}, \norm{u_2}_{X_{\la}^{s,\frac12}}) \norm{w}_{X_{\la}^{s,\frac12}}.
\end{aligned}\]
\end{corollary}

\section{Global well-posedness in $L^2(\T)$: Proofs of Theorems \ref{thm:local0}, \ref{thm:local} and \ref{thm:global}}\label{sec:global}
\subsection{Short proof of Theorem \ref{thm:local0}}\label{sec:WP 12}
Let $s \ge \frac12$ be fixed. We recall the integral equation \eqref{Duhamel} associated to \eqref{eq:kawahara1} in the Fourier space as follows:
\begin{equation}\label{eq:solution map}
\wh{v}(n) = e^{itp_0(n)}\wh{v}_0(n) + \int_0^t e^{i(t-s)p_0(n)} \left(\wh{\N}_{R}(v) + \wh{\N}_{NR}(v)\right)(s,n) \; ds,
\end{equation}
for $\N_R(v)$ and $\N_{NR}(v)$ as in \eqref{eq:resonant term} and \eqref{eq:non-resonant term}, respectively.

\medskip

We denote by $\Gamma(v)$ the map defined as in \eqref{eq:solution map} (after time localization). Then, Lemmas \ref{lem:Xsb}, \ref{lem:resonant1} and \ref{lem:nonres-trilinear} yield
\[\norm{\Gamma (v)}_{Y_T^{s,b}} \le CT^{\frac12-b}\norm{v_0}_{H^s} + CT^{\beta}\norm{v}_{Y_T^{s,b}}^3\]
and
\[\norm{\Gamma (v_1) - \Gamma (v_2)}_{Y_T^{s,b}} \le CT^{\frac12-b}\norm{v_{0,1} - v_{0,2}}_{H^s} + 2CT^{\beta}\left(\norm{v_1}_{Y_T^{s,b}}^2 + \norm{v_2}_{Y_T^{s,b}}^2\right)\norm{v_1-v_2}_{Y_T^{s,b}},\]
for some $0 < \beta= \beta(s) $ and $\frac12 < b = b(s)$ satisfying
\begin{equation}\label{eq:bbeta}
1-2b+\beta < 0.
\end{equation}
Remark that it is possible to choose $b$ and $\beta$ satisfying \eqref{eq:bbeta}, see the proofs of Lemmas \ref{lem:resonant1} and \ref{lem:nonres-trilinear}.
Let $\norm{v_0}_{H^s} \le R$, for a fixed $R >0$. Choosing $T > 0$ satisfying
\[16C^3T^{1-2b+\beta}R^2 \le \frac12,\]
one can show that the map $\Gamma$ is a contraction on the set
\[\set{v \in Y_T^{s,b}: \norm{v}_{Y_T^{s,b}} \le 2CT^{\frac12-b}R},\]
which completes the proof of Theorem \ref{thm:local0}.

\subsection{Proof of Theorem \ref{thm:local}}
The proof of Theorem \ref{thm:local} is based on the standard energy method. We particularly follow the argument in \cite{TT2004}. The scaling argument is not required to be addressed for the regularity $s >0$, indeed, we have $L^4$ estimate without smallness condition, and the trilinear estimate and the smoothing effect are valid for $\la =1$, thus the standard compactness argument ensures the local well-posedness. Therefore, in what follows, we fix $s =0$.

\medskip

For given $R > 0$ and $v_0 \in L^2(\T)$, with $\norm{v_0}_{L^2(\T)} \le R$, we employ the scaling argument: let $v_{\lambda}(t,x) : = \lambda^{-2} v (\lambda^{-5} t, \lambda^{-1} x)$, $\lambda \ge 1$ be a solution to \eqref{eq:scaled kawahara} on $[0,1]$ with $v_{\la}(0,x) =: v_{\la,0}(x) := \la^{-2} v_0(\la^{-1}x)$, then $v$ is the solution to \eqref{eq:kawahara} on $[0,\la^{-5}]$. A straightforward computation gives
\[\norm{v_{\lambda,0}}_{L^2(\T_{\lambda})} = \lambda^{-\frac32}\norm{v_{0}}_{\dot{H}^s(\T)} \le \la^{-\frac32}R.\]
We will choose $\la_0 \gg  1$ at least satisfying $\la^{-\frac32}R \le \rho \ll1$, for all $\la \ge \la_0$, where $\rho$ is to be required in Lemma \ref{lem:L4}. We simply denote $v_{\la}$ and $v_{\la,0}$ by $u$ and $u_0$.

\subsubsection{Existence}\label{sec:local existence}
We first recall \eqref{eq:lambda-kawahara_u}
\begin{equation}\label{eq:kawahara_local}
\begin{aligned}
\pt \wh{u}(n) - ip_{\lambda}(n) \wh{u}(n) =&~{} -\frac{\mu i n}{(2\pi \lambda)^2}(|\wh{u}(n)|^2 - |\wh{u}_0(n)|^2) \wh{u}(n) \\
&+ \frac{\mu i}{3(2\pi \lambda)^2}n \sum_{\N_n^{\lambda}} \wh{u}(n_1)\wh{u}(n_2)\wh{u}(n_3).\\
=:&~{} \wh{\N}_{R}^*(u)(n) + \wh{\N}_{NR}^*(u)(n).
\end{aligned}
\end{equation}

Lemma \ref{lem:Xsb} (after the time localization by multiplying by the smooth cutoff function, but dropping it) gives
\[\norm{u}_{X_{\la}^{0,\frac12}} \lesssim \norm{u_0}_{L_{\la}^2} + \norm{\N_R^*(u)}_{L_t^2L_{\la}^2} + \norm{\N_{NR}^*(u)}_{X_\la^{0,-\frac12+\delta}},\]
for some $0 < \delta \ll 1$. On one hand, Lemma \ref{lem:nonres-trilinear} controls $\norm{\N_{NR}^*(u)}_{X_{\la}^{0,-\frac12+\delta}}$ by
\[\la^{\frac12-}\norm{u}_{X_{\la}^{s,\frac12}}^3.\]
On the other hand, a trivial estimate and Corollary \ref{cor:main} yield
\[\begin{aligned}
\norm{\N_R^*(u)}_{L_t^2L_{\la}^2} \lesssim&~{} \left(\sup_{t \in [-1,1]}\frac{1}{\la^2}\sup_{n \in \Z} |n| \left||\wh{u}(t,n)|^2 - |\wh{u}_0(n)|^2 \right| \right) \norm{u}_{L_2^2L_{\la}^2}\\
\lesssim&~{} \la \log \la\left(\norm{u_0}_{L_{\la}^2}^4 + \big(\norm{u_0}_{L_{\la}^2}^2 + \la^{\frac12}\norm{u}_{X_{\la}^{0,\frac12}}^4\big)^2 +  \la^{-\frac12}\norm{u}_{X_{\la}^{0,\frac12}}^4 + \la^{-\frac12}\norm{u}_{X_{\la}^{0,\frac12}}^6\right) \norm{u}_{X_{\la}^{0,\frac12}}.
\end{aligned}\]
Collecting all, we conclude
\begin{equation}\label{eq:a priori}
\begin{aligned}
\norm{u}_{X_{\la}^{0,\frac12}} \le&~{} C_0\norm{u_0}_{L_{\la}^2} + C_1\la^{\frac12-}\norm{u}_{X_{\la}^{s,\frac12}}^3 \\
&+C_1\la^{0+} \Bigg(\norm{u_0}_{L_{\la}^2}^4 + \big(\norm{u_0}_{L_{\la}^2}^2 + \norm{u}_{X_{\la}^{0,\frac12}}^4\big)^2 +  (1+ \norm{u}_{X_{\la}^{0,\frac12}}^2)\norm{u}_{X_{\la}^{0,\frac12}}^3\Bigg) \norm{u}_{X_{\la}^{0,\frac12}},
\end{aligned}
\end{equation}
for some universal constants $C_0, C_1 > 0$ independent on $\la$.

\medskip

For given $v_0 \in L^2(\T)$, from the density argument, there exists a sequence $\{v_{0}^{(j)}\} \subset C^{\infty}(\T)$ such that $v_{0}^{(j)} \to v_0$ in $L^2(\T)$ as $j \to \infty$. Choose $K= K(R)>0$ such that 
\[\norm{v_{0}^{(j)}}_{L^2}, \norm{v_0}_{L^2} \le K \hspace{1em} \mbox{for all} \hspace{0.5em}j \ge 1.\]
From Theorem \ref{thm:local0} in addition to the energy conservation law \eqref{Hamiltonian}, we have global solutions $v^{(j)}$ to \eqref{eq:kawahara} with initial data $v_{0}^{(j)}$. Choose $\la_1 \ge \la_0 \gg 1$ such that $\la^{-\frac32}K \le \rho$, for all $\la \ge \la_1$. Analogously, \eqref{eq:a priori} is valid for $v_{\la}^{(j)}$ (simply denoted by $u_{j}$, and $u_{0,j}$ denotes $v_{\la,0}^{(j)}$), $j \ge 1$.

\medskip

For each $j \ge 1$, let $p_{\la,j}(n)$ be the Fourier coefficient of the linear operator in \eqref{eq:kawahara_local} with $\frac{\mu n}{(2\pi \lambda)^2} |\wh{v}_{\lambda, 0}^{(j)}(n)|^2$. Let
\[X_j(\la) = \norm{u_j}_{X_{\la,j}^{0,\frac12}},\]
where $X_{\la,j}^{s,b}$ is the $X_{\la}^{s,b}$ space corresponding to $p_{\la,j}(n)$. From the continuity argument, it suffices to show that\footnote{For fixed $\la$, the time localized norm of $X_{\la}^{0,\frac12}$ is continuous in time, thus the claim implies that $X_j(\la) \le 2C_0\la^{-\frac32}K$ on $[-1,1]$.} \\

\medskip

\emph{there exists $\la \gg 1$ so that if $\norm{u_0}_{L_{\la}^2} \le \la^{-\frac32}K$ and $\norm{u}_{X_{\la}^{0,\frac12}} \le 4C_0\la^{-\frac32}K$, then $\norm{u}_{X_{\la}^{0,\frac12}} \le 2C_0\la^{-\frac32}K$}, \\

\medskip

where $C_0 > 0$ is as in \eqref{eq:a priori}. 

\medskip

Then, \eqref{eq:a priori} yields
\[\begin{aligned}
\norm{u}_{X_{\la}^{0,\frac12}} \le&~{} C_0\la^{-\frac32}K +C_1\la^{0+} \Bigg((\la^{-\frac32}K)^4 + \big((\la^{-\frac32}K)^2 +  (4C_0\la^{-\frac32}K)^4\big)^2\\
&+ (1+ (4C_0\la^{-\frac32}K) + (4C_0\la^{-\frac32}K)^3)(4C_0\la^{-\frac32}K)^3\Bigg) (4C_0\la^{-\frac32}K).
\end{aligned}\]
We finally choose $\la \ge \la_1\gg 1$ satisfying
\begin{equation}\label{eq:la cond}
2C_1\la^{0+}(\la^{-\frac32}K)^4 \le \frac18, \quad C_1\la^{\frac12-}(4C_0\la^{-\frac32}K)^2 \le \frac18 \quad \mbox{and} \quad 2C_1\la^{0+}(4C_0\la^{-\frac32}K)^3 \le \frac18
\end{equation}
to prove the claim\footnote{In general, it is natural that $\la^{0+}$ appears in the nonlinear estimate and the smoothing estimate under the $\la$-periodic setting. However, it is controllable, since a suitable norm of the rescaled initial data decays with respect to $\la$ rapidly compare to the growth of $\la^{0+}$.}. Note that \eqref{eq:la cond} is available for the uniqueness part as well.

\medskip

Fix $\la \gg 1$ as above, we now have the uniform bound
\begin{equation}\label{eq:a priori 2}
X_j(\la) \le \delta \ll 1,
\end{equation}
where $\delta:= 2C_0\la^{-\frac32}K$. Moreover, our choice of $\la$ ensures
\begin{equation}\label{eq:initial2}
\norm{u_{0,j}}_{L^2}, \norm{u_0}_{L^2} \le \delta \ll 1 \hspace{1em} \mbox{for all} \hspace{0.5em}j \ge 1.
\end{equation} 
\medskip

The rest of the proof is standard. To close the strong limit argument, we define the Dirichlet projection $\mathbb{P}_k$ for all positive integers $k$ by
\[\mathbb P_k f = \frac{1}{2\pi \la}\sum_{|n| \le k} \wh{f}(n)e^{inx}.\]
Let $u_{j,k} = \mathbb P_ku_j$. Then $u_{j,k}$ satisfies
\[\pt \wh{u}_{j,k}(n) - ip(n) \wh{u}_{j,k}(n) = \sum_{|n| \le k}\left( \wh{\N}_{R}^*(u_j)(n) + \wh{\N}_{NR}(u_j)(n)\right),\]
with $u_{j,k}(0) = \mathbb P_{k}u_{0,j}$. Remark from  \eqref{eq:a priori 2} that 
\begin{equation}\label{eq:a priori 3}
\norm{u_{j,k}}_{X_T^{s,\frac12}} \le \delta \ll 1, \quad j,k \ge 1.
\end{equation}

We are now ready to pass to the (strong) limit. Let $\varepsilon>0$ be given. The proof of  Proposition \ref{prop:main} in addition to Remark \ref{rem:loss of regularity}, \eqref{eq:initial2} and \eqref{eq:a priori 3} ensures
\[\frac{1}{\la^2}\sum_{n \in I_N}|n|\Big||\wh{u}_j(n)|^2-|\wh{u}_{0,j}(n)|^2\Big| \lesssim \la^{0+}N^{0+},\]
which implies
\begin{equation}\label{eq:a priori 4}
\begin{aligned}
\norm{(I-\mathbb P_k)u_j(t)}_{L_{\la}^2}^2 &\lesssim \frac{1}{\la}\sum_{|n| \ge k}|\wh{u}_{0,j}|^2 + \frac{1}{\la}\sum_{N > k}N^{-1}\sum_{n \in I_N}|n|\Big||\wh{u}_j(n)|^2-|\wh{u}_{0,j}(n)|^2\Big| \\
&\le C\left( \norm{(I-\mathbb P_k)u_{0,j}}_{L_{\la}^2}^2 + k^{(-1)-}\la^{1+}\right),
\end{aligned}
\end{equation}
for $C > 0$ and where $I$ is the identity operator. On the other hand, the fact that $u_{0,j} \to u_0$, and \eqref{eq:a priori 4} guarantee that there exists $M > 0$ such that for all $j \ge 1$,
\begin{equation}\label{eq:tightness}
\norm{(I-\mathbb P_k)u_j}_{L_{\la}^2} < \varepsilon
\end{equation}
holds true when $k > M$. Precisely, we can choose $N_0 > 0$ such that $\norm{u_0 - u_{0,j}}_{L_{\la}^2} < \varepsilon/(2\sqrt{2C})$ holds for $j > N_0$. Fix $N_0 > 0$. Then, the $L_{\la}^2$-boundedness of $u_{0,j}$ ensures that for each $1 \le j \le N_0$, there exist $M_j > 0$, $j=1,\cdots, N_0$ such that 
\begin{equation}\label{j<N}
\norm{(I- \mathbb P_k)u_{0,j}}_{L_{\la}^2} < \varepsilon/(\sqrt{2C_0}),  \quad k > M_j, \; 1 \le j \le N_0.
\end{equation}
An analogous argument yields that there exists $M_0 > 0$ such that $k > M_0$ implies $\norm{(I-\mathbb P_k)u_{0}}_{L_{\la}^2} < \varepsilon/(2\sqrt{2C})$ and $k^{(-1)+}\la^{1+} < \varepsilon^2/(2C)$. Let $M := \max(M_0, M_j:1 \le j \le N_0)$ be fixed. Then, for $k > M$, we conclude $\norm{(I-\mathbb P_k)u_{0,j}}_{L_{\la}^2} \le \frac{\varepsilon}{\sqrt{2C_0}}$ for all $j \ge 1$, thanks to \eqref{j<N} and
\[\norm{(I-\mathbb P_k)u_{0,j}}_{L_{\la}^2} \le \norm{u_{0,j}-u_0}_{L_{\la}^2} + \norm{(I-\mathbb P_k)u_{0}}_{L_{\la}^2},\]
which, in addition to \eqref{eq:a priori 4}, implies \eqref{eq:tightness}.

\medskip

Arzel\`a-Ascoli compactness theorem and the diagonal argument yield that for each $\ell \ge 1$, there exists a subsequence $\set{u_{j',j'}} \subset \set{u_{j,k}}$ (denoted by $u_j$) such that
\[\norm{\mathbb P_{\ell}(u_j - u_k)}_{C([-1,1];L_{\la}^2)} \to 0, \quad j,k \to \infty,\]
holds. Therefore, we have a solution $u$ to \eqref{eq:kawahara_local} on $[-1,1]$ satisfying
\begin{equation}\label{eq:solution class}
u \in C([-1,1];L_{\la}^2) \cap X_{\la}^{0,\frac12}, \quad \norm{u}_{X_{\la}^{0,\frac12}} \le \delta \ll 1, \quad \norm{u_j-u}_{C([-1,1];L_{\la}^2)} \to 0, j \to \infty.
\end{equation}

\subsubsection{Completion of the proof of Theorem \ref{thm:local}: Uniqueness, continuity of the flow map and return to \eqref{eq:kawahara}}
Recall \eqref{eq:kawahara3}
\[\pt \wh{w}(n) - ip(n) \wh{w}(n) = \wh{\N}_{R}^*(u_1,u_2,w)(n) + \wh{\N}_{NR}^*(u_1,u_2,w)(n),\]
for $u_1$ and $u_2$ are solutions to \eqref{eq:kawahara_local} satisfying \eqref{eq:solution class} with initial data $u_1(0) = u_2(0) = u_0$, and $w = u-v$. Here $\wh{\N}_{R}^*(u_1,u_2,w)(n)$ and $\wh{\N}_{NR}^*(u_1,u_2,w)(n)$ are explicitly given by
\[\wh{\N}_{R}^*(u_1,u_2,w)(n) = -\frac{\mu i n}{(2\pi \lambda)^2}\left((|\wh{u}_1(n)|^2-|\wh{u}_0(n)^2|)\wh{w}(n) + (|\wh{u}_1(n)|^2 - |\wh{u}_2(n)|^2)\wh{u}_2(n)\right)\]
and
\[\wh{\N}_{NR}^*(u_1,u_2,w)(n) =  \frac{\mu i n}{3(2\pi \lambda)^2} \sum_{\N_n}e^{it\phi(u_0)}\wh{F}(u_1,u_2),\]
for $\wh{F}(u_1,u_2)$ as in \eqref{F}. Similarly as before, the standard $X^{s,b}$ analysis yields
\[\norm{u}_{X_{\la}^{0,\frac12}} \lesssim \norm{u_0}_{L_{\la}^2} + \norm{\N_{R}^*(u_1,u_2,w)}_{L_t^2L_{\la}^2} + \norm{\N_{NR}^*(u_1,u_2,w)}_{X_{\la}^{0,-\frac12+\delta}}.\]
Moreover, Corollaries \ref{cor:main} and \ref{cor:main_uniqueness} and Lemma \ref{lem:nonres-trilinear} yield
\[\begin{aligned}
\norm{\N_{R}^*(u_1,u_2,w)(n)}_{L_t^2L_{\la}^2} \lesssim&~{} \left(\sup_{t \in [-1,1]}\frac{1}{\la^2}\sup_{n \in \Z_{\la}} n\left(|\wh{u}_1(t,n)|^2 - |\wh{u}_0(n)|^2 \right)\right)\norm{w}_{X_{\la}^{0,\frac12}} \\
&+ \left(\sup_{t \in [-1,1]}\frac{1}{\la^2}\sup_{n \in \Z_{\la}} n\left(|\wh{u}_1(t,n)|^2 - |\wh{u}_2(n)|^2 \right)\right)\norm{u_2}_{X_{\la}^{0,\frac12}}\\
\lesssim&~{}+2C_1\la^{0+}(2C_0\la^{-\frac32}K)^3\norm{w}_{X_{\la}^{0,\frac12}}
\end{aligned}\]
and
\[\norm{\N_{NR}^*(u_1,u_2,w)}_{X_T^{s,-\frac12+\delta}} \lesssim \la^{\frac12-}(2C_0\la^{-\frac32}K)^2\norm{w}_{X_{\la}^{0,\frac12}},\]
respectively.

\medskip

Collecting all with \eqref{eq:la cond}, one concludes
\[\norm{w}_{X_{\la}^{0,\frac12}} \le c\norm{w}_{X_{\la}^{0,\frac12}},\]
for some $0 < c < 1$, which implies $w \equiv 0$ on $[-1,1]$.

\medskip

The proof of the continuous dependence of the flow map from initial data to solutions is analogous to the proof of the existence of solutions. The uniqueness of solutions ensures that all convergent subsequences in the sense of \eqref{eq:solution class} have the same limit, thus complete this part.

\subsection{Global well-posedness: Proof of Theorem \ref{thm:global}}
The global well-posedness of \eqref{eq:kawahara} immediately follows from the conservation law 
\begin{equation}\label{eq:conservation}
\int_{\T} v^2(t,x) \; dx = \int_{\T} v^2(0,x) \; dx.
\end{equation}
The conserved quantity \eqref{eq:conservation} ensures $v(\la^{-5}) = v(0)$, for fixed $\la$ chosen in the local theory. Thus, we repeat the local theory on $[\la^{-5}, 2\la^{-5}]$ and further, we, then, obtain the global well-posedness.

\section{Unconditional uniqueness: Proof of Theorem \ref{thm:unconditional}}\label{sec:unconditional}
The aim is to prove that $X_T^{0,\frac12}$ space designed as a solution space in the previous section is large enough to contain $C([0,T];H^s(\T))$, $s > \frac12$ such that the uniqueness in $X_T^{0,\frac12}$ ensures the unconditional  uniqueness in $H^s$, $s > \frac12$ via the interpolation argument.

\medskip

We start with an essential nonlinear estimate in $H^{-s}$.
\begin{lemma}\label{lem:H-s}
Let $s > \frac12$ and $v \in C([0,T];H^s(\T))$ be a solution to \eqref{eq:kawahara la=1}. Then we have
\begin{equation}\label{eq:H-s_1}
\norm{\mathcal F_x^{-1}(n(|\wh{v}(n)|^2 - |\wh{v}_0(n)|^2) \wh{v}(n))}_{L_T^{\infty}H^{-s}} \lesssim \norm{v}_{L_T^{\infty}H^s}^3
\end{equation}
and
\begin{equation}\label{eq:H-s_2}
\norm{\mathcal F_x^{-1}(n \sum_{\N_n} \wh{v}(n_1)\wh{v}(n_2)\wh{v}(n_3))}_{L_T^{\infty}H^{-s}} \lesssim \norm{v}_{L_T^{\infty}H^s}^3.
\end{equation}
\end{lemma}

\begin{proof}
A direct computation gives
\[\mbox{LHS of } \eqref{eq:H-s_1} \lesssim \norm{n^{1-s}(|\wh{v}(n)|^2 - |\wh{v}_0(n)|^2) \wh{v}(n)}_{\ell_n^2} \lesssim \norm{u}_{L_T^{\infty}H^s(\T)}^3,\]
for $s \ge \frac14$. On the other hand, by the duality argument, one reduces the left-hand side of \eqref{eq:H-s_2} as 
\begin{equation}\label{eq:H-s_2.1}
 \sum_{n, \N_n}n^{1-s}\wh{v}(n_1)\wh{v}(n_2)\wh{v}(n_3) \wh{g}(-n),
\end{equation}
for $g \in L^2$ with $\norm{g}_{L^2} \le 1$. Without loss of generality, we assume $|n_1| \le |n_2| \le |n_3|$. We split the summation over frequencies into several cases.

\medskip

\textbf{Case I.} (high $\times$ high $\times$ high $\Rightarrow$ high). We further assume that $|n_1| \sim |n_3| \sim|n|$. Then, the Cauchy-Schwarz inequality yields
\[\begin{aligned}
\eqref{eq:H-s_2.1} \lesssim&~{} \sum_{\substack{n,n_1,n_2 \\ |n_1| \sim |n_2| \sim|n|}} |n|^{1-4s} \wh{f}(n_1)\wh{f}(n_2)\wh{f}(n_3) |\wh{g}(-n)|\\
\lesssim&~{} \left(\sum_{n \in \Z \setminus \set{0}} |n|^{3- 8s} \right)^{\frac12} \norm{f}_{L^2}^3\norm{g}_{L^2}\\
\lesssim&~{} \norm{v}_{L_T^{\infty}H^s}^3,
\end{aligned}\]
for $s > \frac12$, where $\wh{f}(n) = \bra{n}^s|\wh{v}(n)|$.

\medskip

\textbf{Case II-a.} (low $\times$ high $\times$ high $\Rightarrow$ high). We assume that $|n_1| \ll |n_2| \sim |n_3| \sim |n|$. A similar argument yields
\[\begin{aligned}
\eqref{eq:H-s_2.1} \lesssim&~{} \sum_{\substack{n,n_1,n_2 \\ |n_1| \ll |n_2| \sim|n|}} |n|^{1-3s}|n_1|^{-s} \wh{f}(n_1)\wh{f}(n_2)\wh{f}(n_3) |\wh{g}(-n)|\\
\lesssim&~{} \sum_{n_1 \in \Z \setminus \set{0}}\left(\sum_{|n_1| \ll |n|} |n|^{2- 6s} \right)^{\frac12}|n_1|^{-s}\wh{f}(n_1)\norm{f}_{L^2}^2\norm{g}_{L^2}\\
\lesssim&~{} \left(\sum_{n_1 \in \Z \setminus \set{0}} |n_1|^{3- 8s} \right)^{\frac12}\norm{v}_{L_T^{\infty}H^s}^3,
\end{aligned}\]
for $s > \frac12$, which implies the right-hand side of \eqref{eq:H-s_2}.

\medskip

\textbf{Case II-b.} (high $\times$ high $\times$ high $\Rightarrow$ low). Under the condition $|n| \ll |n_1| \sim |n_3|$, an analogous argument ensures
\[\begin{aligned}
\eqref{eq:H-s_2.1} \lesssim&~{} \sum_{\substack{n,n_1,n_2 \\ |n| \ll |n_1| \sim|n_2|}} |n|^{1-s}|n_1|^{-3s} \wh{f}(n_1)\wh{f}(n_2)\wh{f}(n_3) |\wh{g}(-n)|\\
\lesssim&~{} \sum_{n \in \Z \setminus \set{0}}\left(\sum_{|n| \ll |n_1|} |n_1|^{- 6s} \right)^{\frac12}|n|^{1-s}|\wh{g}(-n)|\norm{f}_{L^2}^3\\
\lesssim&~{} \left(\sum_{n \in \Z \setminus \set{0}} |n|^{3- 8s} \right)^{\frac12}\norm{v}_{L_T^{\infty}H^s}^3\norm{g}_{L^2},
\end{aligned}\]
for $s > \frac12$, which implies the right-hand side of \eqref{eq:H-s_2}.

\medskip

\textbf{Case III-a.} (low $\times$ low $\times$ high $\Rightarrow$ high). We assume that $|n_1| \le |n_2| \ll |n_3| \sim |n|$. Since $|n|^{1-2s} \le |n_2|^{1-2s}$, The Cauchy-Schwarz inequality shows
\[\begin{aligned}
\eqref{eq:H-s_2.1} \lesssim&~{} \sum_{\substack{n,n_1,n_2 \\ |n_1| \le |n_2| \ll |n|}} |n_2|^{1-3s}|n_1|^{-s} \wh{f}(n_1)\wh{f}(n_2)\wh{f}(n_3) |\wh{g}(-n)|\\
\lesssim&~{} \sum_{n_1 \in \Z \setminus \set{0}}\left(\sum_{|n_1| \le |n_2|} |n_2|^{2- 6s} \right)^{\frac12}|n_1|^{-s}\wh{f}(n_1)\norm{f}_{L^2}^2\norm{g}_{L^2}\\
\lesssim&~{} \left(\sum_{n_1 \in \Z \setminus \set{0}} |n_1|^{3- 8s} \right)^{\frac12}\norm{v}_{L_T^{\infty}H^s}^3,
\end{aligned}\]
for $s > \frac12$, which implies the right-hand side of \eqref{eq:H-s_2}.

\medskip

\textbf{Case III-b.} (low $\times$ high $\times$ high $\Rightarrow$ low). Using the same argument as before with the fact $|n_2|^{-2s} \le \max(|n_1|,|n|)^{-2s}$, one obtain the right-hand side of \eqref{eq:H-s_2}. We omit the details and complete the proof.
\end{proof}

Now we are ready to prove Theorem \ref{thm:unconditional}. 
\begin{proof}[Proof of Theorem \ref{thm:unconditional}]
Let $T> 0$ be given and $v \in C([0,T];H^s(\T))$, $s > \frac12$ be a solution to \eqref{eq:kawahara_local}. Using the scaling argument (but we here fix larger $\la \gg 1$ satisfying $\la^{-5}T \lesssim 1$ as well.) Then, for $u$ ($\la$-scaled function of $v$), a straightforward calculation 
\begin{equation}\label{Embedding}
\norm{u}_{X_{\la}^{s,0}}
 \lesssim \norm{u}_{L_t^{\infty}H_{\la}^s}
\end{equation}
ensures that the solution $u$ belongs to $X_{\la}^{s,0}$. 

\medskip

On the other hand, Lemma \ref{lem:H-s} in addition to \eqref{Embedding} reveals\footnote{Lemma \ref{lem:H-s} is still available under the $\la$-periodic setting, but with $\la$-dependent bound in the right-hand side.} that 
\[\N_{R}^*(u) + \N_{NR}^*(u) \in X_{\la}^{-s,0},\]
which implies $u$ belongs to $X_{\la}^{-s,1}$ thanks to \eqref{eq:kawahara_local}. 

\medskip

The interpolation theorem for $X_{\la}^{s,0}$ with $X_{\la}^{-s,1}$ ensures $u \in X_{\la}^{0,\frac12}$, that is, the space $C_tH_{\la}^s$ is embedded in the space $X_{\la}^{0,\frac12}$, $s > \frac12$. Therefore, the uniqueness result in $X_{\la}^{0,\frac12}$ established in the previous section (a part of Theorem \ref{thm:local}) guarantees the uniqueness of $u$ in $C_tH_{\la}^s$ (thus, so the uniqueness of $v$ in $C_TH^s$), which completes of the proof of Theorem \ref{thm:unconditional}. 
\end{proof}

\appendix

\section{$L^4$-Strichartz estimates}\label{sec:L4}
The aim of this appendix is to provide the proof of $L^4$-Strichartz estimates (Lemma \ref{lem:L4}) for the sake of reader's convenience. 
\begin{proof}[Proof of Lemma \ref{lem:L4}]
The proof basically follows the proof of Lemma 2.1 in \cite{TT2004} associated the Airy flow. We also refer to \cite{Bourgain1993, KPV1996, Tao2001, Tao2006, Oh, Kwak2018-1, CC2018} for similar arguments. 

\medskip

Let $f = f_1 + f_2$, where 
\[\wh{f}_1(n) = 0, \quad \mbox{if} \quad |n| > 1.\]
Note that $|\{n \in \Z_{\lambda} : n \in \supp (\wh{f}_1)\}| = 2\lambda + 1$. Since $f^2 \le 2f_1^2 + 2f_2^2$, it suffices to treat $\norm{f_1^2}_{L^2(\R \times \T_{\la})}$ and $\norm{f_2^2}_{L^2(\R \times \T_{\la})}$ separately.

\medskip

\textbf{$f_1^2$ case.} A computation gives
\begin{equation}\label{eq:L^4_1}
\norm{f_1^2}_{L^2(\R \times \T_{\la})}^2 \le \int_{\Z_{\la}} \int_{\R} \left| \int_{\Z_{\la}}\int_\R  |\wt{f}_1(\tau_1,n_1)||\wt{f}_1(\tau -\tau_1,n-n_1)|\; d\tau_1 dn_1 \right|^2 \; d\tau d n.
\end{equation}
From the support property, the right-hand side of \eqref{eq:L^4_1} vanishes unless $|n| \le 2$. Let 
\[\wt{F}_1(\tau, n) = \bra{\tau- p_{\la}(n)}^{b}|\wt{f}_1(\tau,n)|.\]
The Cauchy-Schwarz inequality and the Minkowski inequality, we see that for $b > \frac14$,
\[\begin{aligned}
\mbox{RHS of } \eqref{eq:L^4_1} \lesssim&~{}\frac{1}{\la}\sum_{\substack{n \in \Z_{\la} \\ |n| \le 2}} \int_{\R} \Bigg( \int_{\Z_{\la}}\left(\int_\R \bra{\tau-\tau_1 - p_{\la}(n-n_1)}^{-2b}\bra{\tau_1- p_{\la}(n_1)}^{-2b} \; d\tau_1 \right)^{\frac12}\\
& \hspace{7em} \times \left(\int_\R  |\wt{F}_1(\tau_1,n_1)|^2|\wt{F}_1(\tau -\tau_1,n-n_1)|^2\; d\tau_1\right)^{\frac12} dn_1 \Bigg)^2 \; d\tau \\
\lesssim&~{}\frac{1}{\la}\sum_{\substack{n \in \Z_{\la} \\ |n| \le 2}} \Bigg(\int_{\Z_{\la}}\left(\int_{\R^2}  |\wt{F}_1(\tau_1,n_1)|^2|\wt{F}_1(\tau -\tau_1,n-n_1)|^2\; d\tau_1 d \tau \right)^{\frac12}\Bigg)^2\\
\lesssim&~{} \norm{f_1}_{X_{\lambda}^{0,b}}^4 \lesssim \norm{f}_{X_{\lambda}^{0,b}}^4.
\end{aligned}\]

\medskip

\textbf{$f_2^2$ case.} Analogous to \eqref{eq:L^4_1}, we have
\[
\begin{aligned}
\norm{f_2^2}_{L^2(\R \times \T_{\la})}^2 \le&~{} \int_{\Z_{\la}} \int_{\R} \left| \int_{\Z_{\la}}\int_\R  |\wt{f}_2(\tau_1,n_1)||\wt{f}_2(\tau -\tau_1,n-n_1)|\; d\tau_1 dn_1 \right|^2 \; d\tau d n\\
=&~{}\frac{1}{\la}\sum_{\substack{n \in \Z_{\la} \\ |n| \le 1}} \int_{\R} \left| \int_{\Z_{\la}}\int_\R  |\wt{f}_2(\tau_1,n_1)||\wt{f}_2(\tau -\tau_1,n-n_1)|\; d\tau_1 dn_1 \right|^2 \; d\tau d n\\
&+\frac{1}{\la}\sum_{\substack{n \in \Z_{\la} \\ |n| >1}} \int_{\R} \left| \int_{\Z_{\la}}\int_\R  |\wt{f}_2(\tau_1,n_1)||\wt{f}_2(\tau -\tau_1,n-n_1)|\; d\tau_1 dn_1 \right|^2 \; d\tau d n\\
=:&~{} I_1 + I_2.
\end{aligned}\]
The term $I_1$ can be treated similarly as \textbf{$f_1^2$ case}. For the term $I_2$, we may assume that $n_1 > 1$ and $n-n_1>1$ (thus, $n > 1$). Indeed, let $f_2 = f_{2,1} + f_{2,2}$, where
\[\wh{f}_{2,1}(n) = 0 \quad \mbox{if} \quad n > 1,\]
then $\norm{f_2^2}_{L^2}^2 \le 2\norm{f_{2,1}^2}_{L^2}^2 + 2\norm{f_{2,2}^2}_{L^2}^2$ and $\norm{f_{2,1}^2}_{L^2} = \norm{\overline{f_{2,1}}^2}_{L^2} = \norm{f_{2,2}^2}_{L^2}$. Similarly as \eqref{eq:L^4_1}, we have
\[\begin{aligned}
I_2 \lesssim&~{}\frac{1}{\la}\sum_{\substack{n \in \Z_{\la} \\ n > 1}} \int_{\R} \Bigg( \left(\int_{\Z_{\la}}\int_\R \bra{\tau-\tau_1 - p_{\la}(n-n_1)}^{-2b}\bra{\tau_1- p_{\la}(n_1)}^{-2b} \; d\tau_1dn_1 \right)^{\frac12}\\
& \hspace{7em} \times \left(\int_{\Z_{\la}}\int_\R  |\wt{F}_2(\tau_1,n_1)|^2|\wt{F}_2(\tau -\tau_1,n-n_1)|^2\; d\tau_1dn_1\right)^{\frac12}  \Bigg)^2 \; d\tau \\
\lesssim&~{} M\norm{f_2}_{X_{\lambda}^{0,b}}^4,
\end{aligned}\]
where 
\[M = \sup_{\substack{\tau \in \R, n \in \Z_{\la} \\ n > 1}}\int_{\Z_{\la}}\int_\R \bra{\tau-\tau_1 - p_{\la}(n-n_1)}^{-2b}\bra{\tau_1- p_{\la}(n_1)}^{-2b} \; d\tau_1dn_1.\]
Thus, it is enough to show that $M \lesssim 1$ whenever $b > \frac{3}{10}$.

\medskip

By a direct computation
\[\int_{\R} \bra{a}^{-\alpha} \bra{b-a}^{-\alpha} \; da \lesssim \bra{b}^{1-2\alpha},\]
for $\frac12 < \alpha < 1$, we estimate
\[M \lesssim \sup_{\substack{\tau \in \R, n \in \Z_{\la} \\ n > 1}} \frac{1}{\la} \sum_{\substack{n_1 \in \Z_{\la} \\ n_1, n-n_1 > 1}} \bra{\tau-p_{\la}(n_1) - p_{\la}(n-n_1)}^{1-4b}.\]
For each $\tau \in \R$ and $n \in \Z_{\la}$ with $n>1$, we compute
\[\begin{aligned}
&p_{\la}(n_1) + p_{\la}(n-n_1) - \tau \\
=&~{} 5n\left(n_1^4-2nn_1^3 + \left(2n + \frac35\beta \la^{-2}\right)n_1^2 -n\left(n^2 + \frac35\beta\la^{-2}\right)n_1 - C(\tau,n)\right) \\
&- \frac{\mu n_1}{(2\pi \la)^2}|\wh{v}_{\la,0}(n_1)|^2 - \frac{\mu (n-n_1)}{(2\pi \la)^2}|\wh{v}_{\la,0}(n-n_1)|^2,
\end{aligned}\]
where 
\[C(\tau,n) = \frac{\tau - n^5 - \beta \la^{-2} \left(\gamma + \frac{\mu}{2\pi \la}\norm{v_{\la,0}}_{L_{\la}^2}^2\right)n}{5n}.\]
Let 
\[h(x) = x^4-2nx^3 + \left(2n + \frac35\beta \la^{-2}\right)x^2 -n\left(n^2 + \frac35\beta\la^{-2}\right)x - C(\tau,n),\]
then, we know
\[\begin{aligned}
&\frac{d}{dx}h(x) = 4x^3-6nx^2 + 2\left(2n + \frac35\beta \la^{-2}\right)x -n\left(n^2 + \frac35\beta\la^{-2}\right),\\
&\frac{d^2}{dx^2}h(x) = 12x^2 - 12nx + 2\left(2n + \frac35\beta \la^{-2}\right) = 12\left(x + \frac{n}{2}\right)^2 + n + \frac65\beta \la^{-2} > 0.
\end{aligned}\]
This observation reveals that $h(x)$ has the minimum value at $x= \alpha$, where $h'(\alpha) = 0$. If $h(\alpha) \ge 0$, we write
\[h(x) = (x-\alpha)(x^3 + Ax^2 + Bx + C) + h(\alpha),\]
for some $A,B,C$ depending on $\tau, n$. Moreover, since $h'(\alpha) = 0$, $x=\alpha$ is also a zero of $x^3 + Ax^2 + Bx + C =0$, thus
\[h(x) = (x-\alpha)^2(x^2 + A'x + B') + h(\alpha) = (x-\alpha)^2\left(\left(x + \frac{A'}{2}\right)^2 + B' - \frac{A'}{4}\right) + h(\alpha),\]
for some $A',B'$ depending on $\tau, n$. Note that  $B' - \frac{A'}{4} \ge 0$. Let 
\[\Omega_1 = \set{n_1 \in \Z_{\la} : |n_1 - \alpha| \le 1} \quad \mbox{and} \quad \Omega_2 = \set{n_1 \in \Z_{\la} : |n_1 + A'/2| \le 1}.\]
Note that $|\Omega_1 \cup \Omega_2| \le 4\la +2$, and
\[h(n_1) \ge 1 \quad \mbox{in} \left(\Omega_1 \cup \Omega_2\right)^c.\]
Since $n \ge n_1$ and
\[\left|\frac{\mu n_1}{(2\pi \la)^2}|\wh{v}_{\la,0}(n_1)|^2 + \frac{\mu (n-n_1)}{(2\pi \la)^2}|\wh{v}_{\la,0}(n-n_1)|^2\right| \le\frac{n}{2\pi \la} \delta^2 ,\]
we conclude that if $\frac{\delta^2}{2\pi \lambda} \le \frac12$, then
\[p_{\la}(n_1) + p_{\la}(n-n_1) - \tau \ge \frac12n(n-\alpha)^2\left(n + \frac{A'}{2}\right)^2 \quad \mbox{in} \quad \left(\Omega_1 \cup \Omega_2\right)^c,\]
which implies
\[\begin{aligned}
M \lesssim&~{} \frac{1}{\lambda} \sum_{\Omega_1 \cup \Omega_2} 1 +\frac{1}{\la} \sum_{n_1 >1} n_1^{5(1-4b)}\\
&+\frac{1}{\la} \sum_{\Omega_1^c} |n_1 -\alpha|^{5(1-4b)} +\frac{1}{\la} \sum_{\Omega_2^c} |n_1 + A'/2|^{5(1-4b)}\\
\lesssim&~{} 1.
\end{aligned}\]

\medskip

On the other hand, if $h(\alpha) < 0$, $h(x)$ has two zeros so that $h$ is written as
\[h(x) = (x-\sigma_1)(x-\sigma_2)(x^2 + Ax + B),\]
for some $A,B$ depending on $\tau, n$, and $x^2 + Ax + B > 0$ for all $x$. Thus, 
\[h(x) = (x-\sigma_1)(x-\sigma_2)\left(\left(x + \frac{A}{2}\right)^2 + B-\frac{A}{4}\right),\]
where $B-\frac{A}{4} > 0$. Thus, similarly as above, we conclude $M \lesssim 1$ whenever $b > \frac{3}{10}$.
\end{proof}

\section{Weak ill-posedness in $H^s(\T)$, $0 \le s < \frac12$: Proof of Theorem \ref{thm:weak ill-posedness}}\label{sec:ill}
As mentioned in Section \ref{sec:intro}, the proof of Theorem \ref{thm:weak ill-posedness} closely follows Takaoka and Tsutsumi \cite{TT2004}, initially motivated by Burq, G\'erard and Tzvetkov \cite{BGT2002} (for the Schr\"odinger case) and Christ, Colliander and Tao \cite{CCT2003} (not only for the Schr\"odinger but also KdV cases).

\medskip

For the sake of simplicity, we consider the following equation
\[\pt v - \px^5 v + \frac{(2\pi)^2}{3}\left(v^2 - \oint v^2 \right)\px v = 0,\]
equivalently
\begin{equation}\label{eq:SK1}
\pt \wh{v} - i\left(n^5 + n|\wh{v}_0(n)|^2\right) \wh{v} = in(|\wh{v}(n)|^2 - |\wh{v}_0(n)|^2) - \frac{in}{3}\sum_{N_n} \wh{v}(n_1)\wh{v}(n_2)\wh{v}(n_3).
\end{equation}
Duhamel's principle in \eqref{eq:SK1} yields
\begin{equation}\label{eq:SK2}
\begin{aligned}
\wh{v}(n) =&~{} e^{it(n^5 + n|\wh{v}_0(n)|^2)}\wh{v}_0(n) \\
&+ \int_0^t e^{i(t-s)(n^5 + n|\wh{v}_0(n)|^2)} \left(in(|\wh{v}(n)|^2 - |\wh{v}_0(n)|^2) - \frac{in}{3}\sum_{N_n} \wh{v}(n_1)\wh{v}(n_2)\wh{v}(n_3) \right)\; ds.
\end{aligned}
\end{equation}
We denote the second term in the right-hand side of \eqref{eq:SK2} by 
\[\int_0^t \wh{F}(v)(s) \; ds.\]

\medskip

We fix $0 \le s < \frac12$. Let $K$ be a positive integer and set 
\[\varrho_K(x) = \rho\frac{e^{iKx} - e^{-iKx}}{2\pi i},\]
for sufficiently small but fixed $0<\rho \ll 1$\footnote{Here $\rho \ll 1$ ensures the validity of Theorem \ref{thm:local} without the scaling argument.}. We choose initial data $v_{0,K}$ and $v_{0,K}^*$ as follows:
\[\begin{aligned}
&v_{0,K}(x) = K^{-s}\varrho_K(x) \quad \mbox{and}\\
&v_{0,K}^*(x) = K^{-s}(1+\pi K^{2s-1 + \vartheta})^{\frac12}\varrho_K(x) = (1+\pi K^{2s-1 + \vartheta})^{\frac12}v_{0,K}(x),
\end{aligned}\]
for $0 < \vartheta < 1-2s$. A straightforward calculation gives
\[\norm{v_{0,K}}_{H^s}, ~\norm{v_{0,K}^*}_{H^s} \lesssim \rho \ll 1\]
and 
\[\norm{v_{0,K} - v_{0,K}^*}_{H^s} \lesssim \left|1 - (1+\pi K^{2s-1 + \vartheta})^{\frac12}\right| \to 0, \quad \mbox{as} \quad K \to \infty.\]
Taking $t_K = K^{-\vartheta}$, one has 
\[\begin{aligned}
&t_Kn|\wh{v}_{0,K}(n)|^2 = \rho^2\left(\delta_{n K} + \delta_{n -K}\right) K^{1-2s-\vartheta} \quad \mbox{and} \\
&t_Kn|v_{0,K}^*(n)|^2 = \rho^2\left(\delta_{n K} + \delta_{n -K}\right) \left( K^{1-2s-\vartheta} + \pi \right),
\end{aligned}\]
where $\delta_{i j}$ is well-known Kronecker delta.

\medskip

Theorem \ref{thm:local} ensures that there exist $T>0$ and solutions $v_K$ and $v_{K}^*$ to \eqref{eq:kawahara} on $[-T,T]$. We take $K$ sufficiently large such that $t_K < T$. Set $\Phi(K) := K^5$ and define
\[\wh{w}_K(n) = e^{it(\Phi(K)+n|\wh{v}_{0,K}(n)|^2)}\wh{v}_{0,K}(n) - e^{it(\Phi(K)+n|\wh{v}_{0,K}^*(n)|^2)}\wh{v}_{0,K}^*(n).\]
A direct computation gives
\begin{equation}\label{eq:lower bound}\norm{w_K(t_K)}_{H^s}^2 = \left|1-e^{i\pi}(1+2\pi K^{2s-1 + \vartheta})^{\frac12} \right|^2 \ge 4, \quad K \ge 1.
\end{equation}

\medskip

Suppose that the uniform continuity of the flow map holds true, which implies 
\begin{equation}\label{eq:UCD}
\sup_{t \in [-T,T]}\norm{v_K(t) - v_K^*(t)}_{H^s} \to 0, \quad \mbox{as} \quad K \to \infty.
\end{equation}
From \eqref{eq:SK2}, we have
\begin{equation}\label{eq:inner}
(v_K(t_K) - v_K^*(t_K), w_K(t_K))_{H^s}  =~{} \norm{w_K(t_K)}_{H^s} + \left(\int_0^{t_K} F(v_K)(s) - F(v_K^*)(s) ,w_K(t_K)\right)_{H^s},
\end{equation}
where $(\cdot,\cdot)_{H^s}$ is an usual $H^s$ inner product. The standard argument under the periodic setting\footnote{Another auxiliary space based on $\ell_n^2L_{\tau}^1$ is necessary in order to recover the lack of the embedding property ($X^{s,\frac12} \not\hookrightarrow C_tH^s$). Such a space has a property (duality), in our case,
\[\left\|\int_0^t e^{i(t-s)p(-i\px)} w(s) \; ds\right\|_{H_x^s} \lesssim \norm{\bra{n}^s\bra{\tau - p(n)}^{-1} \wt{w}}_{\ell_n^2L_{\tau}^1}.\]
}, in addition to Lemma \ref{lem:nonres-trilinear} and Proposition \ref{prop:main} with the uniform boundedness of $v$ in $X_T^{s,\frac12}$, ensures the second term in the right-hand side of \eqref{eq:inner} (is bounded by $t_K^{\theta'}$ for some $\theta' > 0$ similarly as the right-hand side of \eqref{eq:a priori} without the initial part term, and hence) tends to $0$ as $K \to \infty$. Using \eqref{eq:lower bound} and \eqref{eq:UCD} in \eqref{eq:inner}, one concludes the contradiction, which ends the proof.

\end{document}